\documentclass[final,3p]{elsarticle}

\usepackage{lineno}
\usepackage[colorlinks=true,breaklinks=true]{hyperref}

\usepackage{graphicx}
\usepackage{subfigure}
\usepackage{epstopdf}
\usepackage{epsfig}
\usepackage{amsmath}
\usepackage{amssymb}
\usepackage{amsthm}
\newtheorem{proposition}{Proposition}
\usepackage{booktabs} % For formal tables
\usepackage[ruled]{algorithm2e} % For algorithms

\biboptions{authoryear}

\begin{document}

\begin{frontmatter}

\title{Anisotropic Mesh Filtering by Homogeneous MLS Fitting}

% USE FOR THE REVIEW PROCESS
%\author{paper 24}
\author{Xunnian Yang\\School of Mathematical Sciences, Zhejiang University, China\\Email: yxn@zju.edu.cn}

%% ONLY USE FOR THE FINAL ACCEPTED PAPER SUBMISSION
%\author[first]{First Author}
%\ead{FirstAuthor@email.com}
%\author[last]{Last and Corresponding Author\corref{cor}}
%\ead{LastAuthor@email.com}
%\cortext[cor]{Corresponding author}
%\address[first]{Institution and address of the first author}
%\address[last]{Institution and address of the last author}

\begin{abstract}
  In this paper we present a novel geometric filter, a homogeneous moving least squares fitting-based filter (H-MLS filter), for anisotropic mesh filtering. Instead of fitting the noisy data by a moving parametric surface and projecting the noisy data onto the surface, we compute new positions of mesh vertices as the solutions to homogeneous least squares fitting of moving constants to local neighboring vertices and tangent planes that pass through the vertices. The normals for defining the tangent planes need not be filtered beforehand but the parameters for balancing the influences between neighboring vertices and neighboring tangent planes are computed robustly from the original data under the assumption of quadratic precision in each tangent direction. The weights for respective neighboring points for the least squares fitting are computed adaptively for anisotropic filtering. The filter is easy to implement and has distinctive features for mesh filtering. (1) The filter is locally implemented and has circular precision, spheres and cylinders can be recovered exactly by the filter. (2) The filtered mesh has a high fidelity to the original data without any position constraint and salient or sharp features can be preserved well. (3) The filter can be used to filter meshes with various kinds of noise as well as meshes with highly irregular triangulation.
\end{abstract}


\end{frontmatter}

% Comment out for final accepted paper submission
%\linenumbers

%%%%%%%%%%%%%%%%%%%%%%%%%%%%%%%%%%%%%%%%%%%%%%%%%%%%%%%%%%%%%%%%%%%%%%%%%%%%%%%%%
%% Section 1: Introduction
%%%%%%%%%%%%%%%%%%%%%%%%%%%%%%%%%%%%%%%%%%%%%%%%%%%%%%%%%%%%%%%%%%%%%%%%%%%%%%%%%

\section{Introduction}
\label{Sec:Introduction}
Along with the development of 3D scanning devices, triangular meshes have
been widely used to represent surface shapes in the fields of geometric
modeling, computer graphics and computer vision. Due to the accuracy
limitations of data capturing, data loss caused by storage or transmission,
etc., geometric models represented by triangular meshes can include noise.
Thus, constructing a visually smooth triangular mesh from a noisy input one plays
a fundamental role for many disciplines and applications.

\begin{figure}[tpb]
  \centering
  % Requires \usepackage{graphicx}
  \subfigure[]{\includegraphics[width=3.6cm]{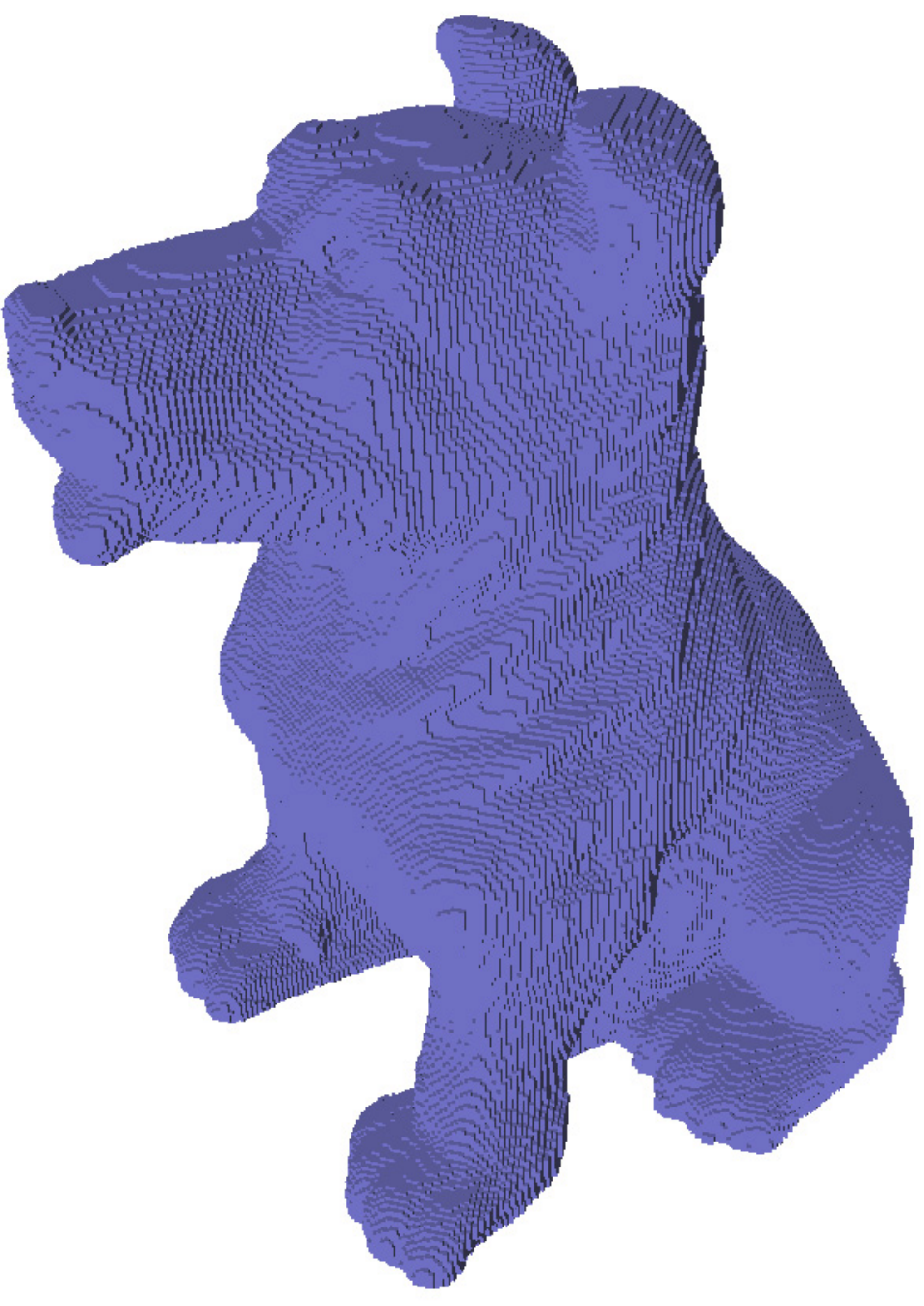}}
  \subfigure[]{\includegraphics[width=3.6cm]{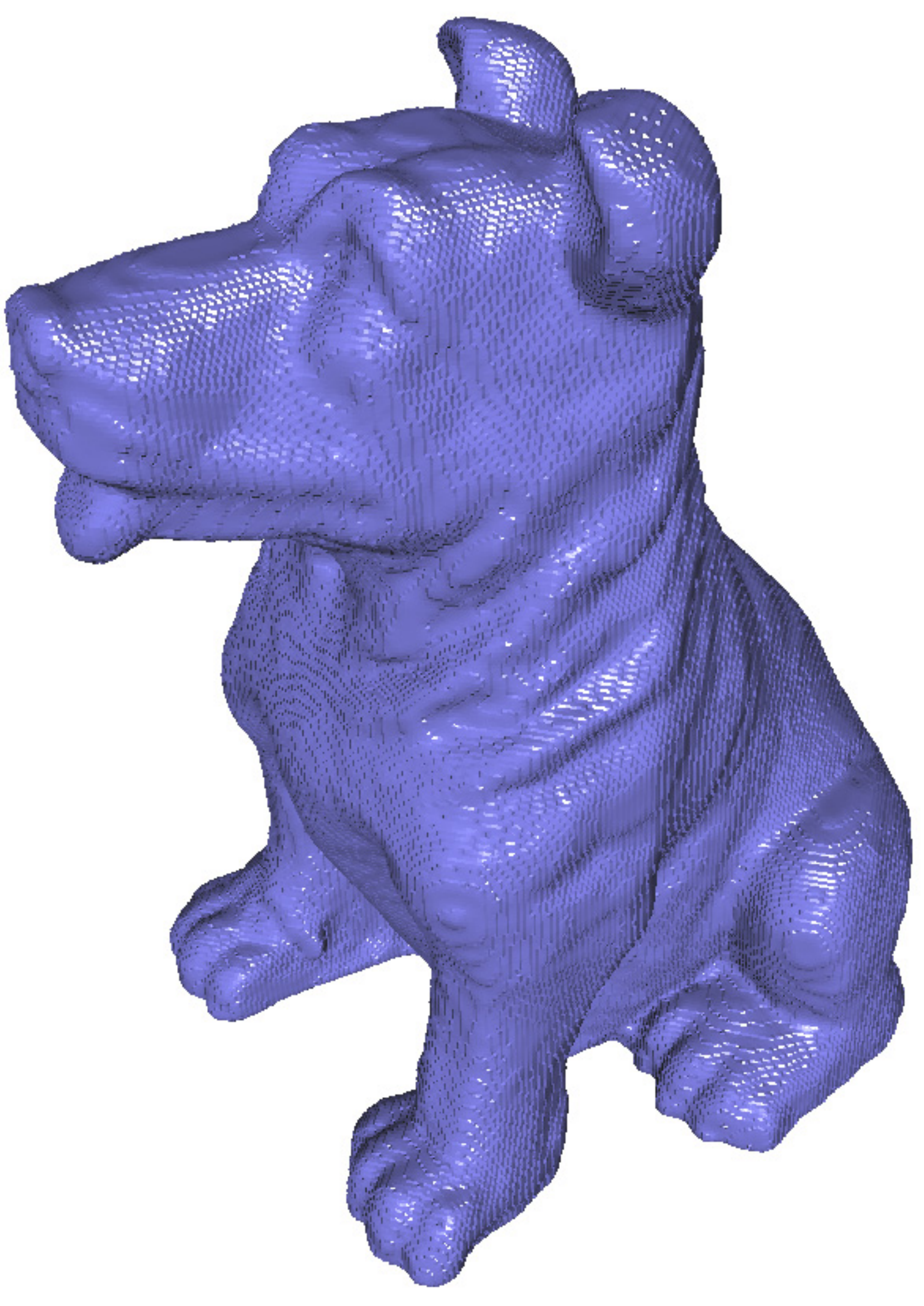}}
  \subfigure[]{\includegraphics[width=3.6cm]{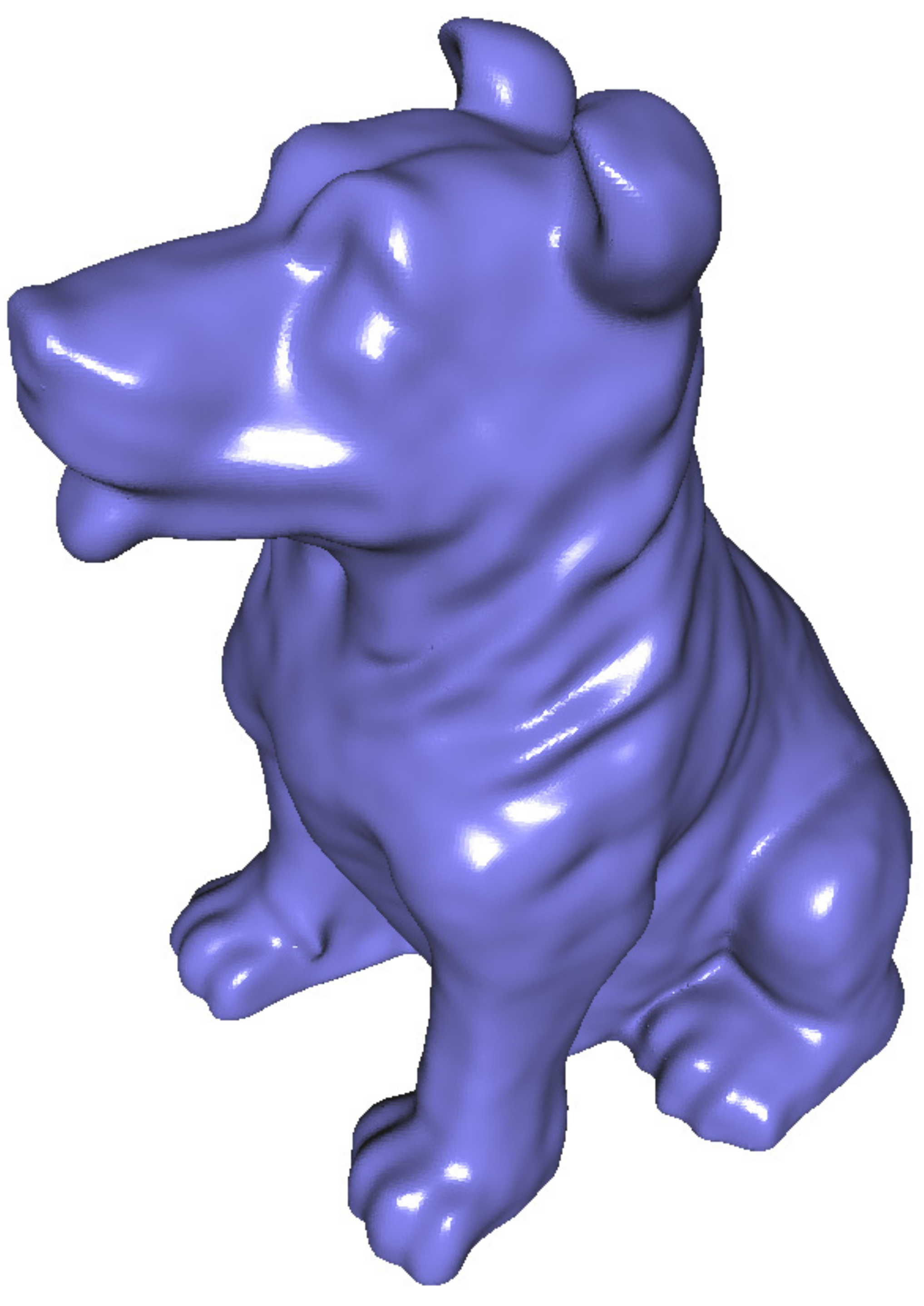}}
  \subfigure[]{\includegraphics[width=3.6cm]{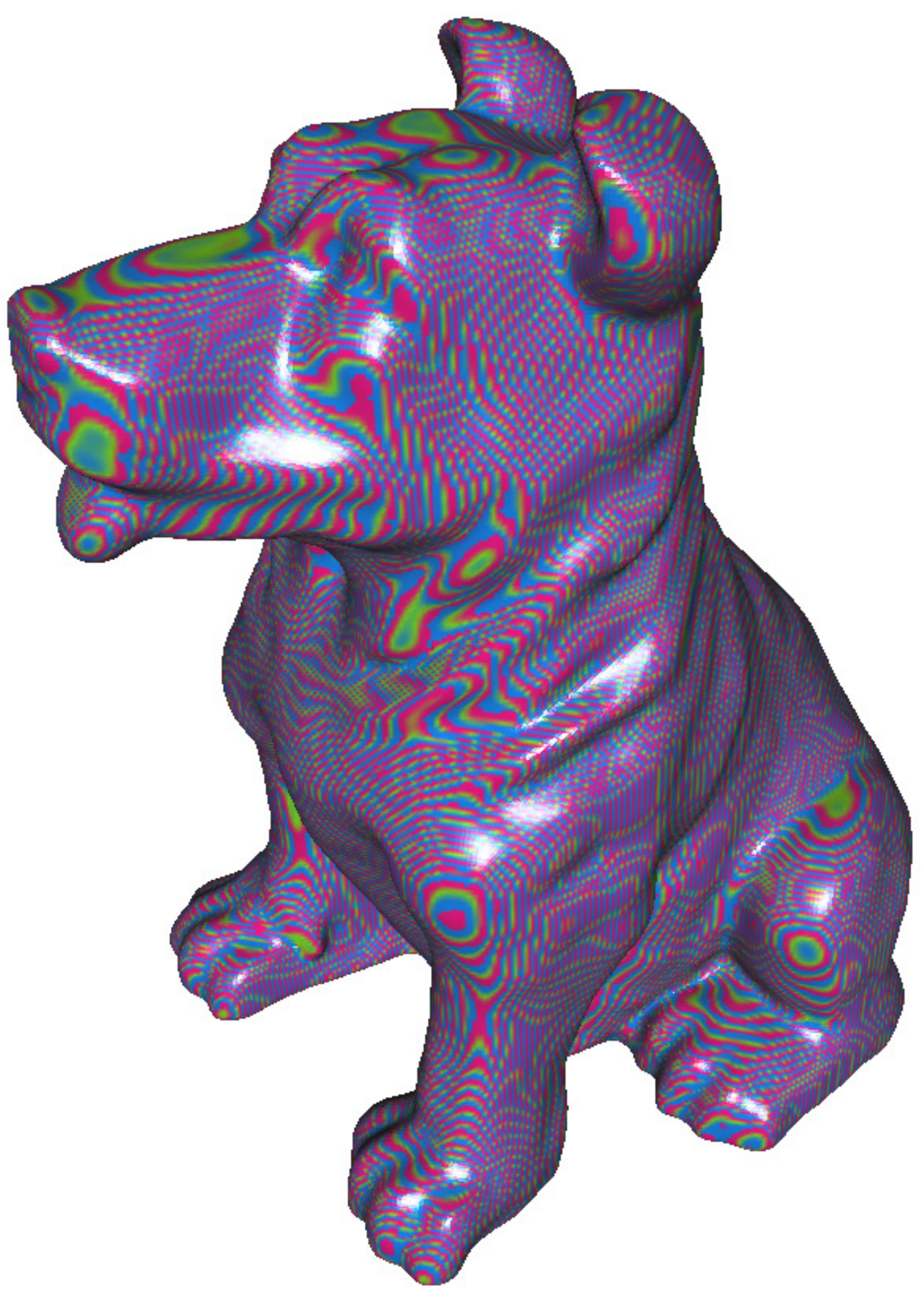}}
  \caption{Mesh filtering by H-MLS filter: (a) the surface mesh of a binary voxel model;
  (b)\&(c) meshes after 2 or 5 iterations of filtering; (d) error plot. The purple, green
  and cyan colors in (d) represent the positive, zero or negative distances from the noisy
  vertices to the filtered surface. In this paper all noisy and filtered meshes are flat shaded if there is no special declaration.}
  \label{Fig:Filtering cubicdog_noisy}
\end{figure}

Many mesh filtering algorithms have been invented by adapting filters for
signal processing or techniques for continuous surface smoothing. However,
filtering a noisy mesh is more challenging than filtering a 1D or 2D signal
or fairing a continuous surface represented by explicit parametric (usually
spline) patches. It is more difficult to distinguish sharp features from a
noisy mesh than on a piecewise smooth surface. The triangulation of a surface
mesh may be highly irregular, direct definition of discrete differential
operators on a triangular mesh may be inaccurate. Minor deformation of a
filtered signal or image is nearly invisible, but the deformation of a
surface mesh can be noticed clearly.

In this paper we present a novel geometric filter for mesh filtering.
Figure~\ref{Fig:Filtering cubicdog_noisy} illustrates an example of mesh
filtering by our proposed H-MLS filter. Given a triangulated surface
mesh of a binary voxel model, a high fidelity smooth surface is obtained just
by a few iterations of filtering without any position constraint. Stairs on
original noisy surface have been removed effectively while salient features
are preserved perfectly after filtering.

Our proposed technique is motivated by moving least squares (MLS) surface
construction from geometric Hermite data~\citep{Yang2016:MatrixWeightRational}.
Instead of constructing a smooth parametric surface explicitly, we compute
new position of each mesh vertex as the solution to the homogeneous least squares fitting of
moving constants to neighboring vertices, tangent planes at the vertices as
well as a line passing through or near the original noisy vertex.
In order to achieve a high accuracy of approximation, a local shape parameter is introduced to balance the influences
between neighboring vertices and neighboring tangent planes. Weights for
respective neighboring vertices or planes are used to adapt or control local
surface features.

When the filtered vertices are obtained as the solutions to homogeneous moving least squares
fitting, the filtering effects depend on the parameters and weights heavily.
We use data dependent parameters for high fidelity anisotropic mesh filtering.
In order to attenuate the influences of data noise, we compute weights and parameters
for least squares fitting from the local geometric data directly under the
assumption of quadratic precision in each tangent direction. This filter is
flexible enough to recover vertices in smooth surface regions with low or
high curvatures as well as vertices lying on sharp edges or corners.

Briefly, the main contributions of the paper are as follows.
\begin{itemize}
  \item We filter mesh vertices by least squares fitting of moving constants to generalized Hermite data, and the new positions of vertices are directly obtained.
  \item The weights and parameters for the least squares fitting are chosen for high accuracy approximation or feature preserving filtering.
  \item The proposed filter is simple but very powerful for filtering meshes with various kinds of noise or triangulation.
\end{itemize}

The paper is organized as follows. In Section~\ref{Section:previous work} we
give a brief review and analysis of existing methods for mesh filtering or
denoising. The motivation of the proposed mesh filter is introduced in
Section~\ref{Section:MLS surface fitting} and
Section~\ref{Section:point-normal filter} presents the algorithm steps of the
new filter. Experimental results are given in Section~\ref{Section:results}
and the paper is concluded in Section~\ref{Section:conclusion} with a brief
summary and discussion.

%-------------------------------------------------------------------------
\section{Previous work} \label{Section:previous work}

Lots of algorithms have been proposed for mesh filtering or mesh denoising in
the past two decades. According to the basic principles for mesh denoising,
the existing algorithms can be roughly classified into four categories.

$\bullet$ \textbf{Signal processing-based methods.} Under the assumption that
noises are high frequency signals, many low-pass filters or statistics-based
filters for signal or image processing have been generalized to process
surface meshes~\citep{Tau95,Alexa02,Peng2001,Fleishman03,Diebel06,Sun08,Jones03}.
These methods can usually give visually smooth results, but the challenges
are how to distinguish noise from sharp features or how to avoid local or
global deformation of the filtered meshes. Practically, a prior step for
recognizing features from the noisy data is used for feature-preserving mesh
denoising~\citep{Shimizu2005,Fan10,BianTong2011,WeiLPWLW17:TensorVotingGuidedMeshDenoising}.

$\bullet$ \textbf{PDE or optimization methods.} Based on the assumption that
a triangular mesh is a discrete approximation of a smooth surface and vertex
noise can cause locally high curvatures, techniques based on geometric flows
or anisotropic diffusion have been developed for mesh fairing or
feature-preserving mesh
filtering~\citep{BaXu03,Clar00,Polthier04,ZhangBenHamza2007}. By using $l_1$
or $l_0$ norm of surface variations as the smoothness measure, sharp features
or piecewise flat patches can be recovered~\citep{HeSchaefer13:L0meshdenoising, WuZCF15:L1meshdenoising}.
When a PDE model has been used for mesh smoothing, proper definition of
discrete differential operators plays an important role for the smoothing
results~\citep{Desbrun99,Polthier07}. Meshes faired by PDE models can have
high quality but may deviate from the original surfaces, one should carefully
choose parameters to balance the surface quality and data
fidelity~\citep{CentinS18:MeshdenoisingwithGeoMetricFidelity}.

$\bullet$ \textbf{Normal-guided or data-driven methods.} Piecewise smooth
normal mapping can help to restore smooth surfaces with feature preserving.
Techniques for filtering facet normals and refitting surface meshes to the
filtered normals have been studied extensively in recent
years~\citep{ZhangWuZD15:variationMeshNormalDenoising,
ZhangDZBL15:guidedNormalFiltering, Shen04, Sun07, ZhengFuTai2011,
YadavRP18:Normalvotingtensorbinaryoptim}. Tasdizen, et al.~(\citeyear{Tasd03})
proposed to filter surface normals by a diffusion equation and refit an
implicit surface to the filtered normals. An interesting way is to filter
facet normals by data
driven-methods~\citep{WangLT16:cascadednormalregression}, of which the effects
depend much on the quality and amount of training data. Normal filtering does
work for mesh denoising, but it is actually a chicken-and-egg problem to
filter normals for mesh denoising. Normals of surface meshes constructed from
medical image data or binary voxel models may not be filtered correctly by
conventional methods. How to avoid flipped edges when refitting triangles to
the filtered normals is also a challenge.

$\bullet$ \textbf{MLS surface fitting methods.} Moving least-squares fitting
can be used to construct smooth surfaces from noisy input data and the
original noisy data can also be filtered by projecting points onto the MLS
surface~\citep{LipmanCL07:data-dependentMLSapproximation,FleishmanCS05:MLSfitting,Huangetal2013EdgeAwarePointSetResampling}.
For mesh filtering purposes, especially for filtering meshes constructed from
medical image data, Wang et al.~(\citeyear{WangY11:meshsmoothingvialocalsurfacefitting}) proposed to filter
mesh vertices by fitting a local quadratic surface to neighborhood of every
vertex. Similarly, Wei et al.~(\citeyear{WeiZYWPWQH15:smoothbinaryvolumes})
employed biquadratic B\'{e}zier surfaces for data fitting and point
projection. Because a linear system has to be solved for fitting each local
surface, how to guarantee the numerical stability and how to improve the
computational efficiency have to be considered seriously.

Instead of fitting selected points by moving low-order surface patches, our proposed mesh filter is derived by homogeneous least-squares fitting of moving constants to selected geometric Hermite data. The solutions to the least squares fitting of moving constants can be obtained directly with no need of solving any linear system. By properly chosen parameters for the least squares fitting, points on low or high curvature regions can be restored effectively. Moreover, points on sharp edges or surface meshes with highly irregular triangulation can also be restored very well.

%-------------------------------------------------------------------------
\section{Homogeneous MLS Fitting to Geometric Hermite Data}\label{Section:MLS surface fitting}

Our proposed mesh filter is motivated by a technique of constructing homogeneous MLS (H-MLS) surfaces
from point-normal pairs~\citep{Yang2016:MatrixWeightRational}.
Differently from conventional MLS surfaces which fit moving low-order
polynomials to points, H-MLS surfaces constructed from point-normal pairs fit
moving constants to points and tangent planes at the points. Constructing H-MLS surfaces from
point-normal pairs needs not solving any linear system and the obtained
surfaces can have circular precision in each tangent direction. Due to its
simplicity and high accuracy, the point evaluation scheme for H-MLS
surfaces constructed from point-normal pairs will be modified for mesh
filtering.

Assume $\{\mathbf{p}_i\}_{i=1}^N$ are the vertices of a polygon or a surface
mesh and $\{\mathbf{n}_i\}_{i=1}^N$ are unit normal vectors at the points.
Assume $\{\xi_i\}_{i=1}^N$ are the corresponding parameters on a 1D or 2D
domain. Let $\phi(\xi)=\phi(|\xi|)$ be a non-negative, symmetric and
decreasing function. An H-MLS curve or surface that fits the points and the planes passing through the points with given
normals can be obtained by minimizing the following objective function
\begin{equation}\label{Eqn:LS-surfaceObjective}
  F(\mathbf{p})=\frac{1}{2}\sum_{i=1}^N\phi_i(\xi)(\mathbf{p}-\mathbf{p}_i)^2
  +\frac{1}{2}\sum_{i=1}^N\mu_i\phi_i(\xi)[(\mathbf{p}-\mathbf{p}_i)\cdot \mathbf{n}_i]^2,
\end{equation}
where $\phi_i(\xi)=\phi(|\xi-\xi_i|)$, $i=1,2,\ldots,N$. The
parameters $\mu_i$s are used to balance the influences of points and
lines/planes that pass through the points. When all the
parameters satisfy $\mu_i>-1$, it is guaranteed that the objective functional
$F(\mathbf{p})$ is convex. By minimizing $F(\mathbf{p})$ or solving the
equation $\frac{\partial F(\mathbf{p})}{\partial \mathbf{p}}=0$, a curve or
surface that fits the original data is obtained as
\[
\mathbf{p}=\mathbf{p}(\xi)=\left[\sum_{i=1}^N
M_i\phi_i(\xi)\right]^{-1}\sum_{i=1}^N M_i\mathbf{p}_i\phi_i(\xi),
\]
where $M_i=I+\mu_i\mathbf{n}_i\mathbf{n}_i^T$, $i=1,2,\ldots,N$, and $I$ is the
identity matrix. Figure~\ref{Fig:MLS_torus}(a) illustrates an H-MLS surface
constructed from point-normals which were originally sampled from a torus.
Particularly, the points are uniformly parametrized and $\phi(\xi)$ is chosen as
the standard bicubic B-spline function. By choosing all $\mu_i=1.25$ for
defining the functional (\ref{Eqn:LS-surfaceObjective}), the obtained H-MLS
surface lies closely to the sampled points. As a comparison, if we choose all
$\mu_i=0$ within the functional, the H-MLS surface reduces to a bicubic
B-spline surface; see Figure~\ref{Fig:MLS_torus}(b).

\begin{figure}[htb]
\centering
  \subfigure[]{\includegraphics[width=4.0cm]{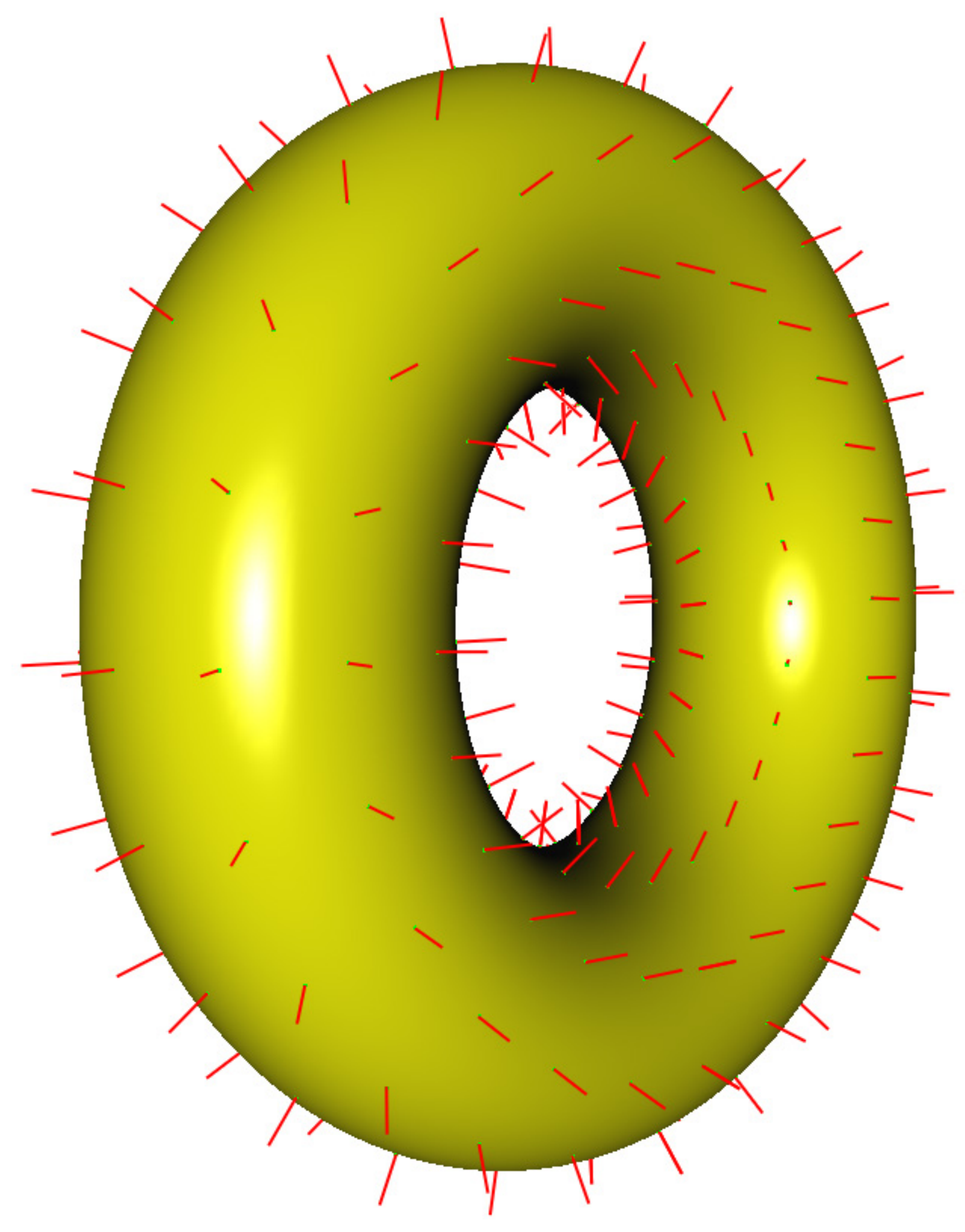}}
  \subfigure[]{\includegraphics[width=4.0cm]{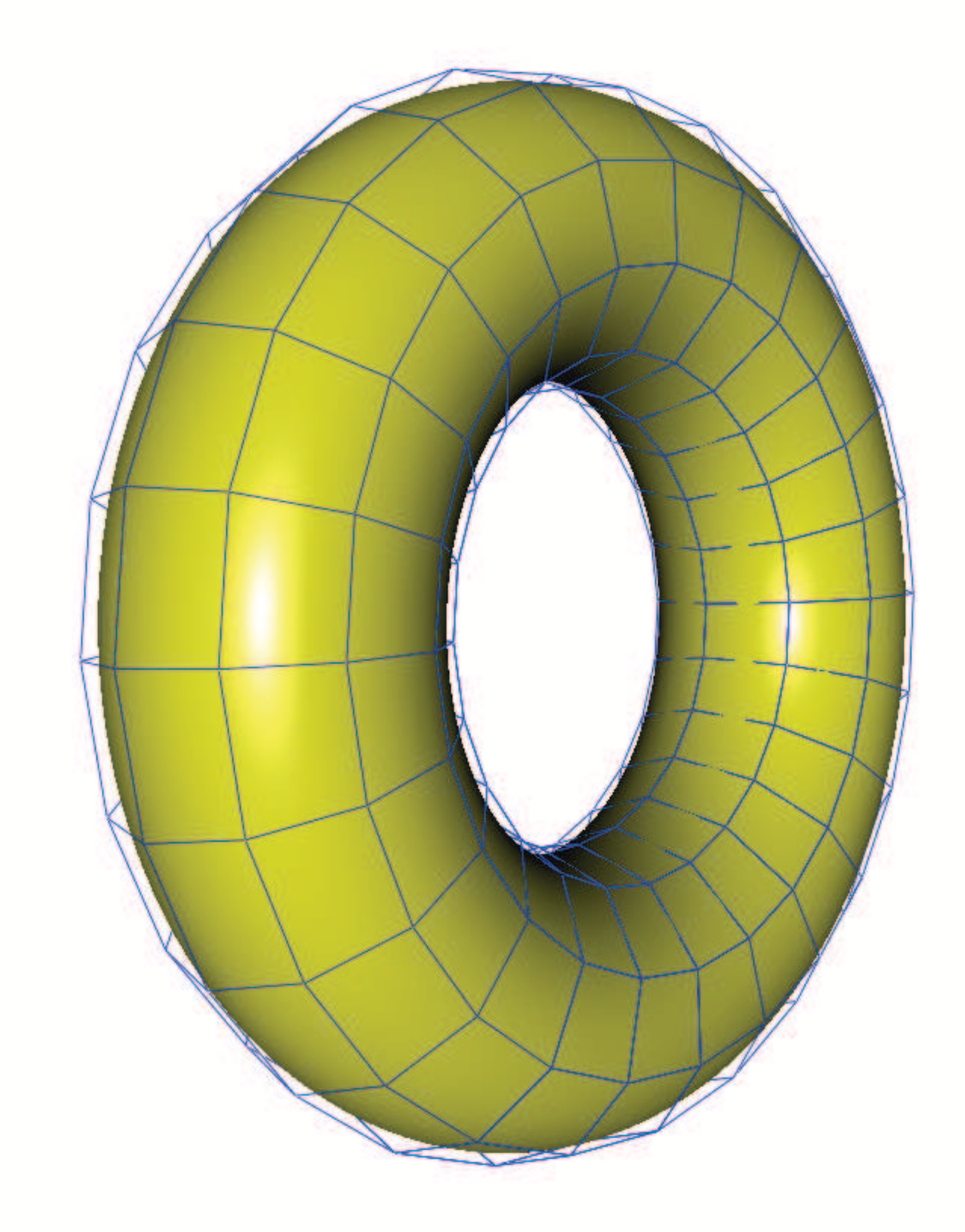}}
  \caption{(a) An H-MLS surface constructed from $10\times20$ point-normal pairs ($\mu_i=1.25$).
  (b) A bicubic B-spline surface with the same set of control points.}
  \label{Fig:MLS_torus}
\end{figure}

Assume the kernel function $\phi(\xi)$ is locally supported, i.e.,
$\phi(\xi)=0$ when $|\xi|>R$, evaluation of point at a parameter $\xi_i$ on
the curve or the surface can be simplified. Denote $N(i)=\{j|
|\xi_i-\xi_j|\leq R\}$. The point at $\xi=\xi_i$ is computed by
\begin{equation}\label{Eqn:LS-surfacePoint}
\hat{\mathbf{p}}_i=\mathbf{p}(\xi_i)=\left[\sum_{j\in
N(i)}\omega_{ij}M_j\right]^{-1}\sum_{j\in N(i)} \omega_{ij}M_j\mathbf{p}_j,
\end{equation}
where $\omega_{ij}=\phi(|\xi_i-\xi_j|)$. Equation
(\ref{Eqn:LS-surfacePoint}) also implies that $\hat{\mathbf{p}}_i$ is the
solution to a local least squares fitting with weights $\omega_{ij}$ and
balance parameters $\mu_j$, $j\in N(i)$.

By computing the parameters $\mu_j$ under the assumption that the curve $\mathbf{p}(\xi)$ has almost circular or helix precision to the original curve from which the points and normals are sampled, and because smooth curves can be approximated well by smoothly connected circular arcs or helix segments,
Equation (\ref{Eqn:LS-surfacePoint}) has been successfully applied for fairing a polygon or a rational spline curve in 2D pr 3D space~\citep{Yang2018:fairingMatNURBScurve}.
In this paper we generalize this technique for filtering surface meshes. However, at least three challenges have to be solved or
avoided. First, computing a global parametrization of a general surface mesh
is a time consuming step for mesh filtering. Second, the neighborhood of a
selected point $\mathbf{p}_i$ may not be symmetric and the naively evaluated
point may shift on or near the tangent plane that passes through the point.
Third, how to preserve salient or sharp features and how to control the
fidelity from the filtered mesh to the original one are more challenging.

%-------------------------------------------------------------------------
\section{H-MLS Filter for Anisotropic Mesh Filtering}
\label{Section:point-normal filter}

In this section we generalize Equation (\ref{Eqn:LS-surfacePoint}) to filter
general types of triangular meshes. Though the weights $\omega_{ij}$ are
theoretically dependent on the parametrization of the vertices, we will
replace $|\xi_i-\xi_j|$ by $\|\mathbf{p}_i-\mathbf{p}_j\|$ or other similar
metrics when special properties like feature preserving are desired. In the
following we introduce necessary steps and implementation details for our
proposed mesh filter.

Before computing new positions for vertices, normal vectors at all mesh vertices are first
estimated. Let $N_T(i)$ denote the index set of triangles that share a vertex $\mathbf{p}_i$. The unit normal vector at the point $\mathbf{p}_i$ is computed by
\[
\mathbf{n}_i=\frac{\sum_{j\in N_T(i)} \alpha_{ij} \mathbf{n}_{f_j}}{\|\sum_{j\in N_T(i)} \alpha_{ij} \mathbf{n}_{f_j}\|},
\]
where $\mathbf{n}_{f_j}$ are the unit normals at the neighboring triangles and $\alpha_{ij}$ are the angles of the triangles at the vertex $\mathbf{p}_i$.

$\bullet$ \textbf{Compute the filtered vertex.}
We compute the filtered vertex by solving a homogeneous MLS fitting equation. Let $N(i)$ denote the index set of selected neighboring points of vertex
$\mathbf{p}_i$. We modify the functional~(\ref{Eqn:LS-surfaceObjective}) to
compute optimal new position for vertex $\mathbf{p}_i$ as follows
\begin{equation}\label{Eqn:LS-VertexObjective}
 % \begin{array}{lcl}
  F_i(\mathbf{p})=\frac{1}{2}\sum_{j\in N(i)}\omega_{ij}(\mathbf{p}-\mathbf{p}_j)^2
  + \frac{\mu_i}{2}\sum_{j\in N(i)}\omega_{ij}[(\mathbf{p}-\mathbf{p}_j)\cdot \mathbf{n}_j]^2
  + \frac{\gamma}{2}[(I-\mathbf{n}_i\mathbf{n}_i^T)(\mathbf{p}-\mathbf{p}_i^*)]^2.
 % \end{array}
\end{equation}
The first two terms on the right side of Equation
(\ref{Eqn:LS-VertexObjective}) represent the squared distances from a point
$\mathbf{p}$ to a set of neighboring points as well as the tangent
planes at the points. The third term means the squared distance from point
$\mathbf{p}$ to the line that passes through $\mathbf{p}_i^*$ with direction
$\mathbf{n}_i$. This term can help to keep the filtered vertex lying close to
the specified line even when the neighborhood is not symmetric. In this paper
we choose $\mathbf{p}_i^*=\mathbf{p}_i$ and $\gamma=1000$ for mesh filtering
by default. Figure~\ref{Fig:foot_denoise_lineconstraint} shows an example of
how line constraint can help to preserve the basic triangle shapes of original
mesh during filtering. Let $\mathbf{p}_i^{center}$ be the centroid of 1-ring
neighborhood of $\mathbf{p}_i$. We can also choose
$\mathbf{p}_i^*=\mathbf{p}_i^{center}$ for vertex filtering, which will lead
to a smooth mesh with more uniform triangulation in the end.

\begin{figure}[htb]
  \centering
  % Requires \usepackage{graphicx}
  \includegraphics[width=3.0cm]{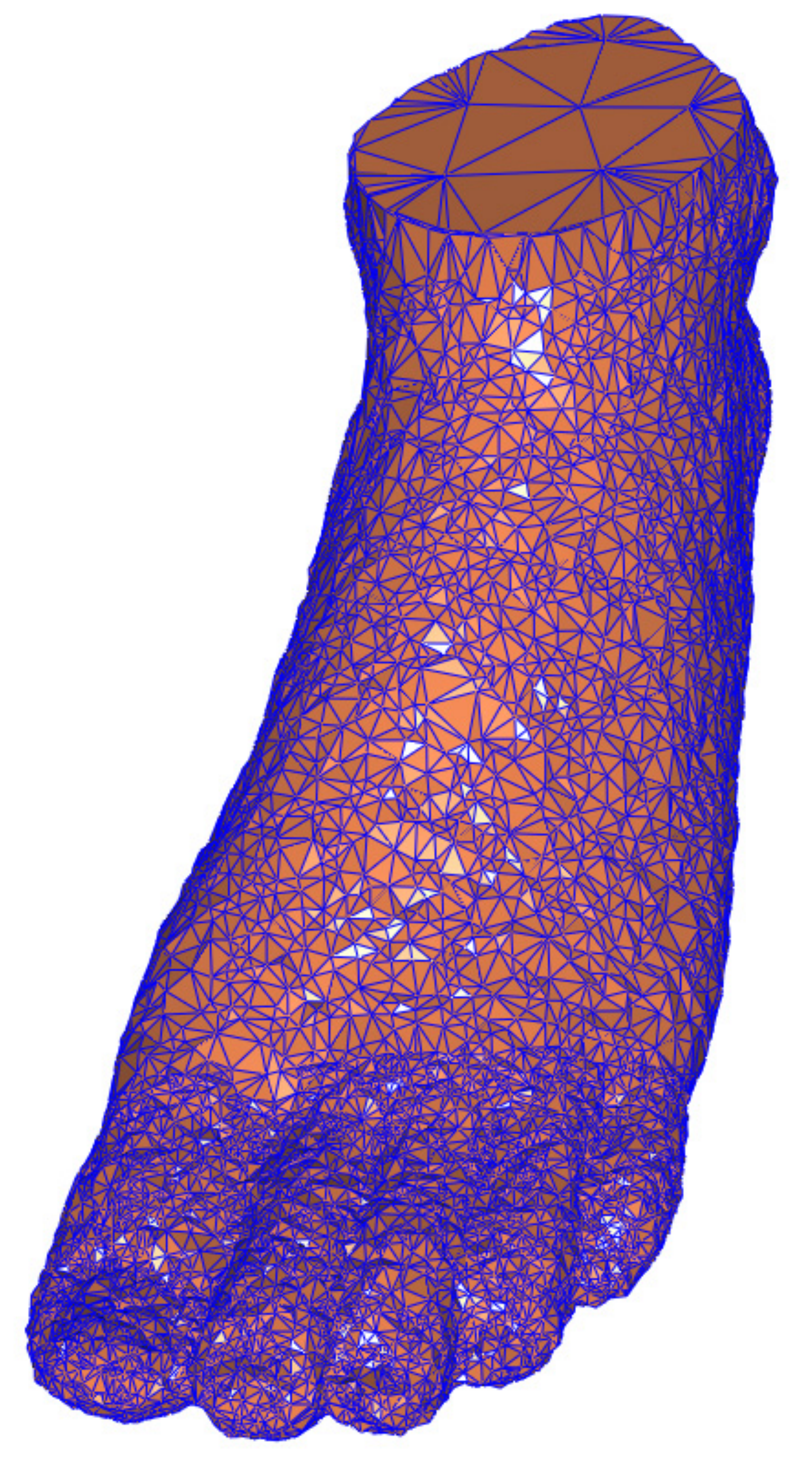}
  \includegraphics[width=3.0cm]{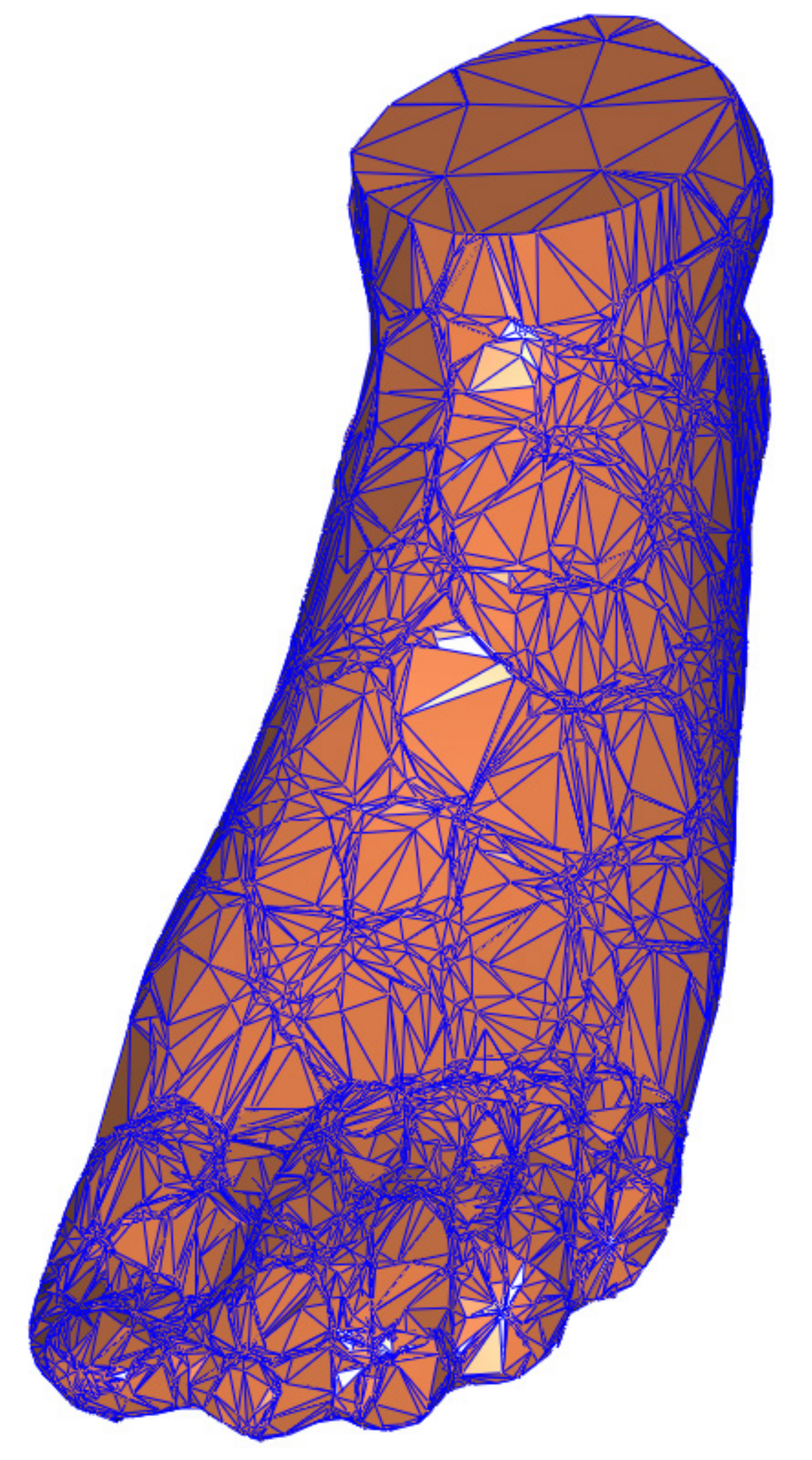}
  \includegraphics[width=3.0cm]{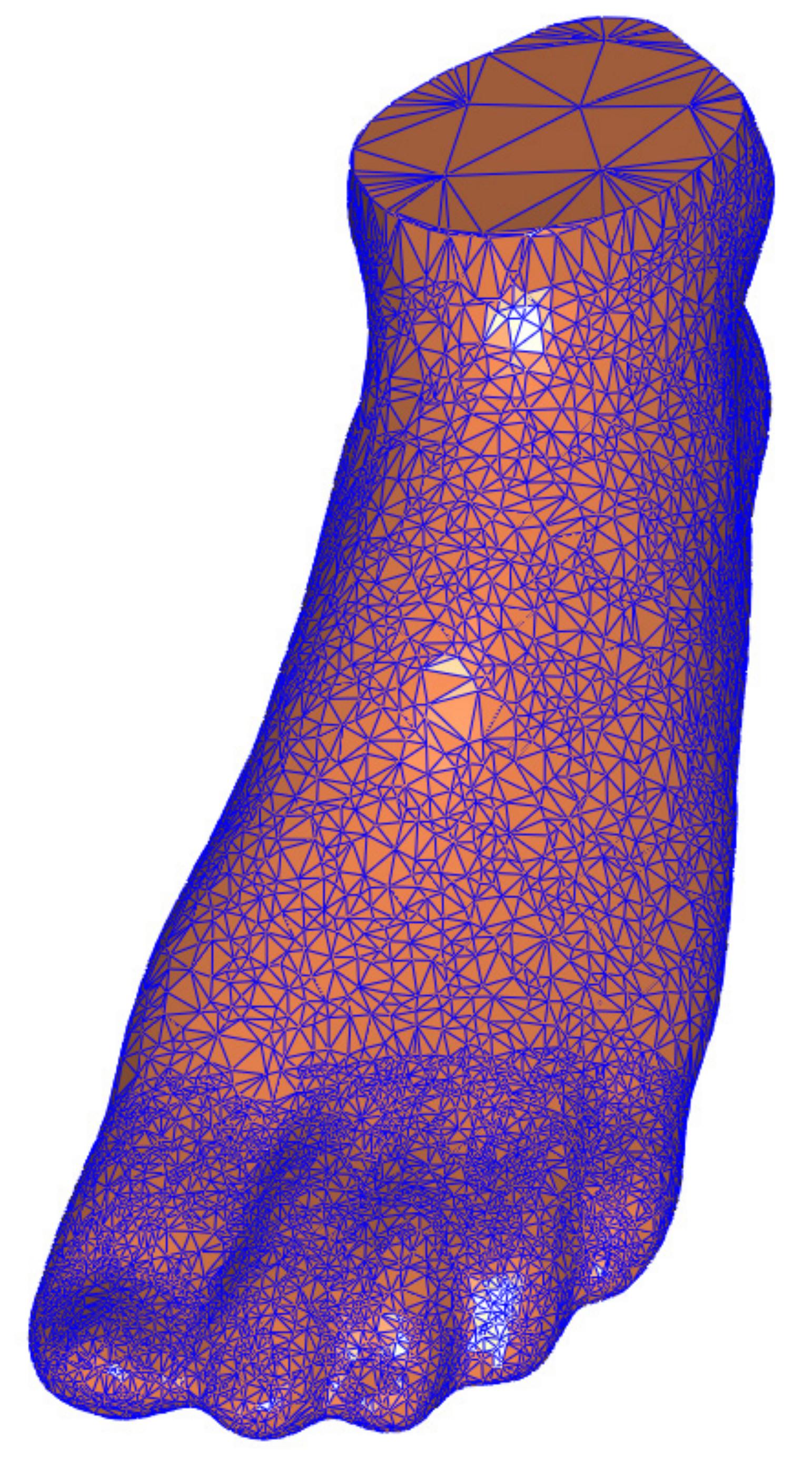}
  \caption{Left: the input noisy mesh; middle \& right: meshes filtered by 3 iterations of H-MLS filtering
  without or with constraint of lines through vertices. }
  \label{Fig:foot_denoise_lineconstraint}
\end{figure}

Under the assumption that all weights $\omega_{ij}$ and the parameter $\mu_i$ are
already known, the derivative of the functional $F_i(\mathbf{p})$ with
respect to the variable $\mathbf{p}$ is
\begin{equation}\label{Eqn:derivative of F(p)}
%  \begin{array}{lcl}
  \frac{\partial F_i(\mathbf{p})}{\partial \mathbf{p}}=\sum_{j\in N(i)}\omega_{ij}(\mathbf{p}-\mathbf{p}_j)
   + \mu_i\sum_{j\in N(i)}\omega_{ij}\mathbf{n}_j\mathbf{n}_j^T(\mathbf{p}-\mathbf{p}_j)
   + \gamma(I-\mathbf{n}_i\mathbf{n}_i^T)(\mathbf{p}-\mathbf{p}_i^*).
%  \end{array}
\end{equation}
Let $M_{ij}=\omega_{ij}(I+\mu_i\mathbf{n}_j\mathbf{n}_j^T)$,
$j\in N(i)$, and $M_i^*=\gamma(I-\mathbf{n}_i\mathbf{n}_i^T)$. The solution
to equation $\frac{\partial F_i(\mathbf{p})}{\partial \mathbf{p}}=0$ is
\begin{equation}\label{Eqn:LS-filteredPoint}
    \hat{\mathbf{p}}_i=\left[\sum_{j\in N(i)}M_{ij} +
    M_i^*\right]^{-1}\left(\sum_{j\in N(i)}M_{ij}\mathbf{p}_j +
    M_i^*\mathbf{p}_i^*\right).
\end{equation}
We assert that the inverse matrix within Equation (\ref{Eqn:LS-filteredPoint}) can be computed robustly by
choosing proper weights and parameters.
\begin{proposition} The matrix
$\sum_{j\in N(i)}M_{ij}+M_i^*$ within Equation (\ref{Eqn:LS-filteredPoint})
is reversible when the weights and parameters satisfy $\omega_{ij}\geq0$,
$j\in N(i)$ but $\sum_{j\in N(i)}\omega_{ij}^2>0$, $\mu_i>-1$ and
$\gamma\geq0$.
\end{proposition}
\begin{proof}
Notice that the matrix
$\sum_{j\in N(i)}M_{ij}+M_i^*$ is symmetric. To prove the matrix is
reversible, we should just prove that the matrix is positive definite. Assume
$\mathbf{x}$ is an arbitrary non-zero vector, a quadratic form is obtained as
\[
\begin{array}{l}
    \mathbf{x}^T\left[\sum_{j\in N(i)}M_{ij}+M_i^*\right]\mathbf{x}\\
    =\sum_{j\in N(i)}\omega_{ij}(\mathbf{x}^T\mathbf{x}+\mu_i\mathbf{x}^T\mathbf{n}_j\mathbf{n}_j^T\mathbf{x})
    +\gamma(\mathbf{x}^T\mathbf{x}-\mathbf{x}^T\mathbf{n}_i\mathbf{n}_i^T\mathbf{x})\\
    =\sum_{j\in N(i)}\omega_{ij}(\|\mathbf{x}\|^2+\mu_i|\mathbf{x}^T\mathbf{n}_j|^2)
    +\gamma(\|\mathbf{x}\|^2-|\mathbf{x}^T\mathbf{n}_i|^2)\\
    >0.
\end{array}
\]
Since the matrix $\sum_{j\in N(i)}M_{ij}+M_i^*$ is positive definite, it is
reversible. The proposition is proven.
\end{proof}
Though an inverse matrix has to be computed for filtering each vertex, the computational cost has
not been increased significantly as compared with state of art mesh denoising
methods. The inverse of a matrix of order 3 can be computed directly by
definition. As explained later, Equation (\ref{Eqn:LS-filteredPoint}) will be
used as an efficient filter for mesh filtering by computing proper weights
and parameters from the original data.

$\bullet$ \textbf{Compute the parameter $\mu_i$.} From Equations
(\ref{Eqn:LS-VertexObjective}) and (\ref{Eqn:LS-filteredPoint}) we know that
the parameter $\mu_i$ plays a key role for computing a high fidelity filtered
vertex. The mesh vertex may be over smoothed when $\mu_i$ has been chosen a
small value. A large value of the parameter may push the filtered vertex
toward the tangent planes of neighboring vertices. The parameter $\mu_i$
should be adaptive to local shape to prevent the filtered mesh from local
deformation but it should also be robust enough against data noise for mesh
filtering.

We propose to compute the value of $\mu_i$ based on the assumption that
$\hat{\mathbf{p}}_i=\mathbf{p}_i$ when the mesh is noise free, irrespective
of the point $\mathbf{p}_i$ lying on a smooth region or on a sharp feature.
Denote by $F'_i(\mathbf{p}_i)$ the derivative of the functional $F_i(\mathbf{p})$
at $\mathbf{p}=\mathbf{p}_i$. The scalar product between $\mathbf{n}_i$ and
$F'(\mathbf{p}_i)$ is obtained as
\[
%\begin{array}{lcl}
    \mathbf{n}_i^T F'_i(\mathbf{p}_i)
    =\sum_{j\in N(i)}\omega_{ij}\mathbf{n}_i^T(\mathbf{p}_i-\mathbf{p}_j)
    +\mu_i\sum_{j\in N(i)}\omega_{ij}\mathbf{n}_i^T\mathbf{n}_j\mathbf{n}_j^T(\mathbf{p}_i-\mathbf{p}_j).
%\end{array}
\]
From equation $\mathbf{n}_i^T F'_i(\mathbf{p}_i)=0$, we have
\begin{equation}\label{Eqn:mu_i_noisy}
    \mu_i=\frac{\sum_{j\in N(i)}\omega_{ij}\mathbf{n}_i^T(\mathbf{p}_i-\mathbf{p}_j)}{\sum_{j\in N(i)}\omega_{ij}\mathbf{n}_i^T\mathbf{n}_j\mathbf{n}_j^T(\mathbf{p}_j-\mathbf{p}_i)}.
\end{equation}
Though the parameter $\mu_i$ computed by above equation can be
used to recover vertices on a noise-free mesh, the values of $\mu_i$ may vary
much from vertex to vertex when the mesh data are noisy. A noise sensitive
parameter cannot be used for mesh filtering.

To compute a robust value for the parameter $\mu_i$, every term in Equation
(\ref{Eqn:mu_i_noisy}) should be computed in a robust way. For a surface mesh
the normal curvature at vertex $\mathbf{p}_i$ along an edge
$\mathbf{p}_i\mathbf{p}_j$ can be estimated by an osculating arc that
interpolates the end points and normal vector $\mathbf{n}_i$ at
$\mathbf{p}_i$, i.e.
$k_{ij}^n=\frac{2\mathbf{n}_i^T(\mathbf{p}_i-\mathbf{p}_j)}{\|\mathbf{p}_i-\mathbf{p}_j\|^2}$;
see, for example~\citep{Meyer03}. It yields that
$\mathbf{n}_i^T(\mathbf{p}_i-\mathbf{p}_j)=\frac{1}{2}k_{ij}^n\|\mathbf{p}_i-\mathbf{p}_j\|^2$.
Similarly, we have
$\mathbf{n}_j^T(\mathbf{p}_j-\mathbf{p}_i)=\frac{1}{2}k_{ji}^n\|\mathbf{p}_j-\mathbf{p}_i\|^2$.
If $\mathbf{p}_i$ and $\mathbf{p}_j$ lie closely on a curvature continuous
surface, the normal curvatures at the two points satisfy $k_{ij}^n\approx
k_{ji}^n$. It follows that $\mathbf{n}_i^T(\mathbf{p}_i-\mathbf{p}_j)\approx
\mathbf{n}_j^T(\mathbf{p}_j-\mathbf{p}_i)$. For a mesh vertex lying on a
non-convex region the normal curvatures to neighboring vertices may have
different signs. Even if the mesh is free of noise, the algebraic sum of the
denominator within Equation (\ref{Eqn:mu_i_noisy}) may still be close to zero
or vanish. To keep the denominator from being a small value or zero and to
compute the parameter against data noise, we replace both the terms
$\mathbf{n}_i^T(\mathbf{p}_i-\mathbf{p}_j)$ and
$\mathbf{n}_j^T(\mathbf{p}_j-\mathbf{p}_i)$ within Equation
(\ref{Eqn:mu_i_noisy}) by
\begin{equation}\label{Eqn:d_ij}
    d_{ij}=\max\{\frac{1}{2}(|\mathbf{n}_i^T(\mathbf{p}_i-\mathbf{p}_j)|+|\mathbf{n}_j^T(\mathbf{p}_j-\mathbf{p}_i)|),\eta\},
\end{equation}
where $\eta>0$ is a small given number. In our experiments we
choose $\eta=0.001l_e$, where $l_e$ is the average edge length of the mesh.
By the same way, we replace the term $\mathbf{n}_i^T\mathbf{n}_j$ within
Equation (\ref{Eqn:mu_i_noisy}) by
$c_{ij}=\max\{\mathbf{n}_i^T\mathbf{n}_j,0.001\}$ which also keeps the
denominator positive. Now, the parameter $\mu_i$ can be computed robustly as
\begin{equation}\label{Eqn:mu_i_robust}
    \mu_i=\frac{\sum_{j\in N(i)}\omega_{ij}d_{ij}}{\sum_{j\in N(i)}\omega_{ij}c_{ij}d_{ij}}.
\end{equation}

The parameter computed by Equation (\ref{Eqn:mu_i_robust}) is definitely
positive and it is equal to the value computed by Equation
(\ref{Eqn:mu_i_noisy}) when every pair of normal curvatures $k_{ij}^n$ and
$k_{ji}^n$ are equal with each other. Thus, a mesh vertex can be recovered
exactly by Equation (\ref{Eqn:LS-filteredPoint}) when its neighboring
vertices and normals are uniformly sampled from a sphere or cylinder. This
property is important and a filtered surface mesh almost suffers no shrinkage
or deformation even without any position constraint.

$\bullet$ \textbf{Choose kernels and compute weights.} As our proposed mesh
filtering scheme can be considered a discrete analogy of kernel based MLS
surface construction from an arbitrary topology triangular mesh that may have
noisy vertices and noisy normals, the smoothness of the filtered mesh depends
heavily on the properties of the kernel function. In this
paper we choose Gaussian function $\phi(\xi)=\exp(-\frac{\xi^2}{2\sigma^2})$
as the basic kernel function.

Based on the chosen kernel function, there are two ways to compute the
weights for vertex filtering.
(1) \emph{Weights for isotropic filtering}. As
discussed in Section~\ref{Section:MLS surface fitting} and the start of this
section, a direct way to compute the weights is
$\omega_{ij}=\exp(-\frac{\|\mathbf{p}_i-\mathbf{p}_j\|^2}{2\sigma_r^2})$,
where $\sigma_r$ is a user specified number. This choice of weights is useful
to remove noise, but sharp features of the mesh are also smoothed out.
(2) \emph{Weights for anisotropic filtering}. To preserve salient or sharp
features during filtering, we can choose the weights
$\omega_{ij}=\exp(-\frac{d_{ij}^2}{2\sigma_s^2})$, where $d_{ij}$ are
computed by Equation (\ref{Eqn:d_ij}) and $\sigma_s$ is another parameter
specified by users. From its definition we know that
$d_{ij}=\frac{1}{4}(|k_{ij}^n|+|k_{ji}^n|)\|\mathbf{p}_i-\mathbf{p}_j\|^2$.
By this choice of weights, a mesh is smoothed much more along a feature
direction (usually with absolute minor curvature) than across a feature
direction. As a result, salient or even sharp features can be preserved well
after vertex filtering. In this paper we use the second choice of weights for
anisotropic mesh filtering.

\begin{figure}[htb]
  \centering
  % Requires \usepackage{graphicx}
  \includegraphics[width=3.0cm]{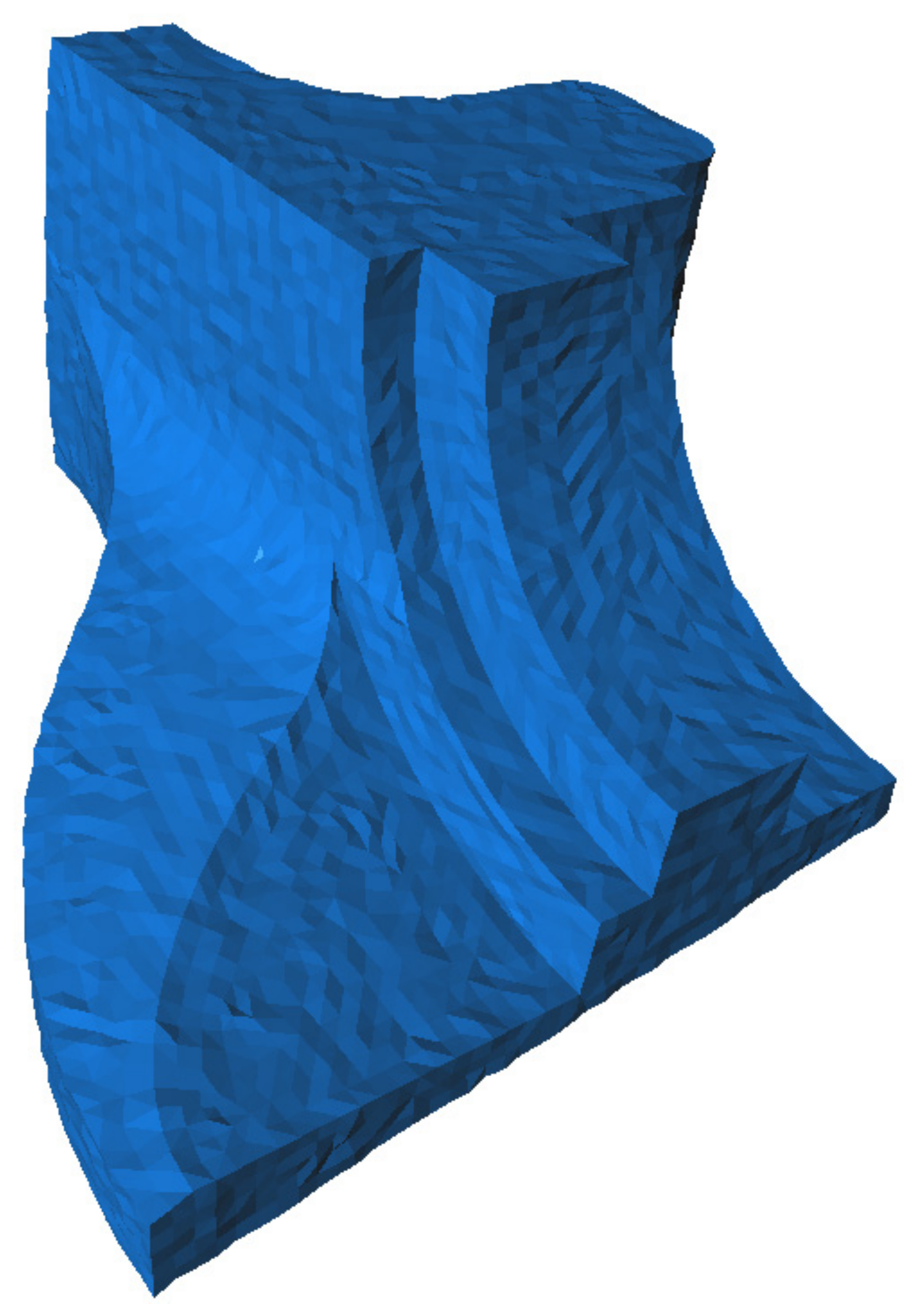}
  \includegraphics[width=3.0cm]{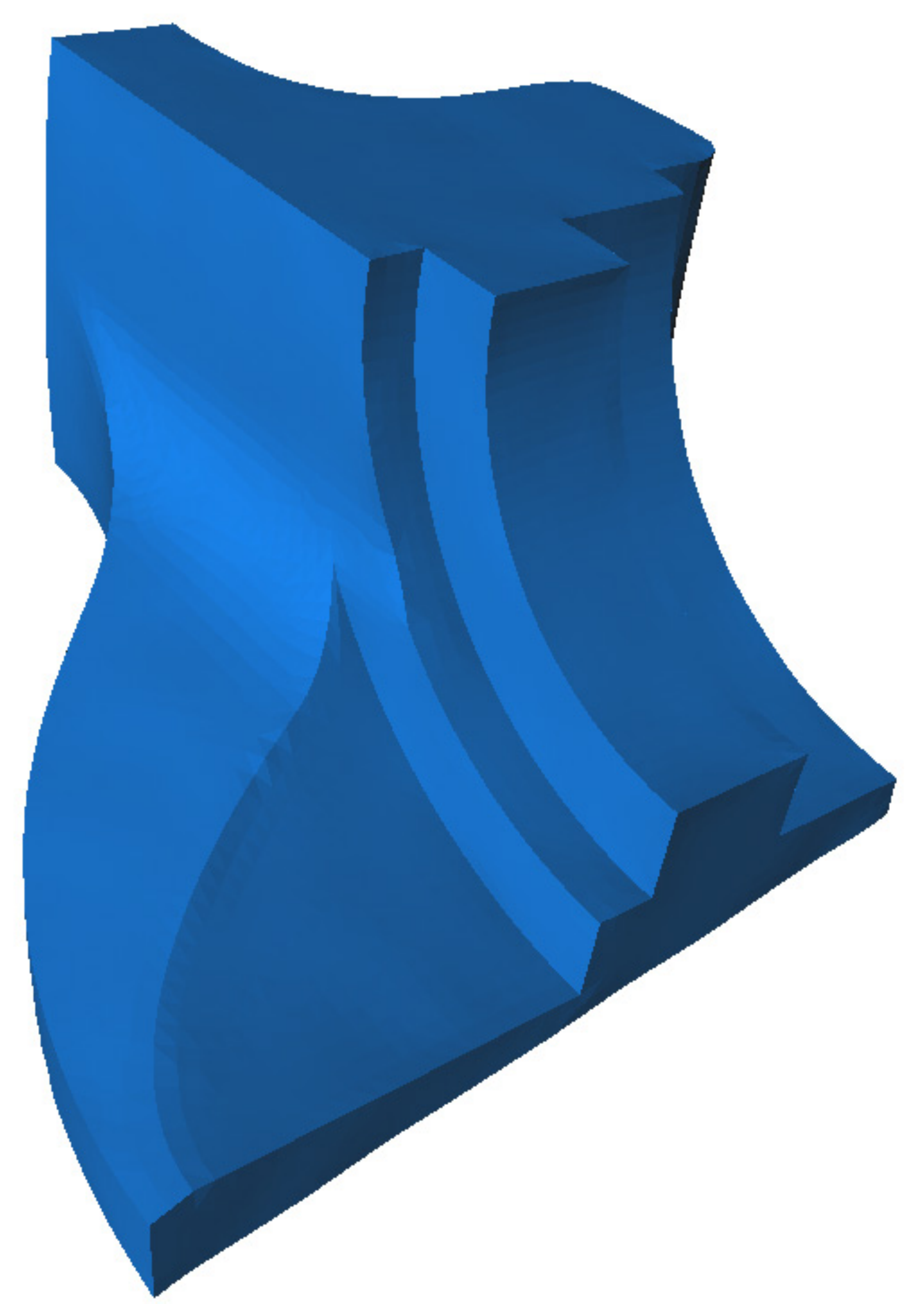}
  \includegraphics[width=3.0cm]{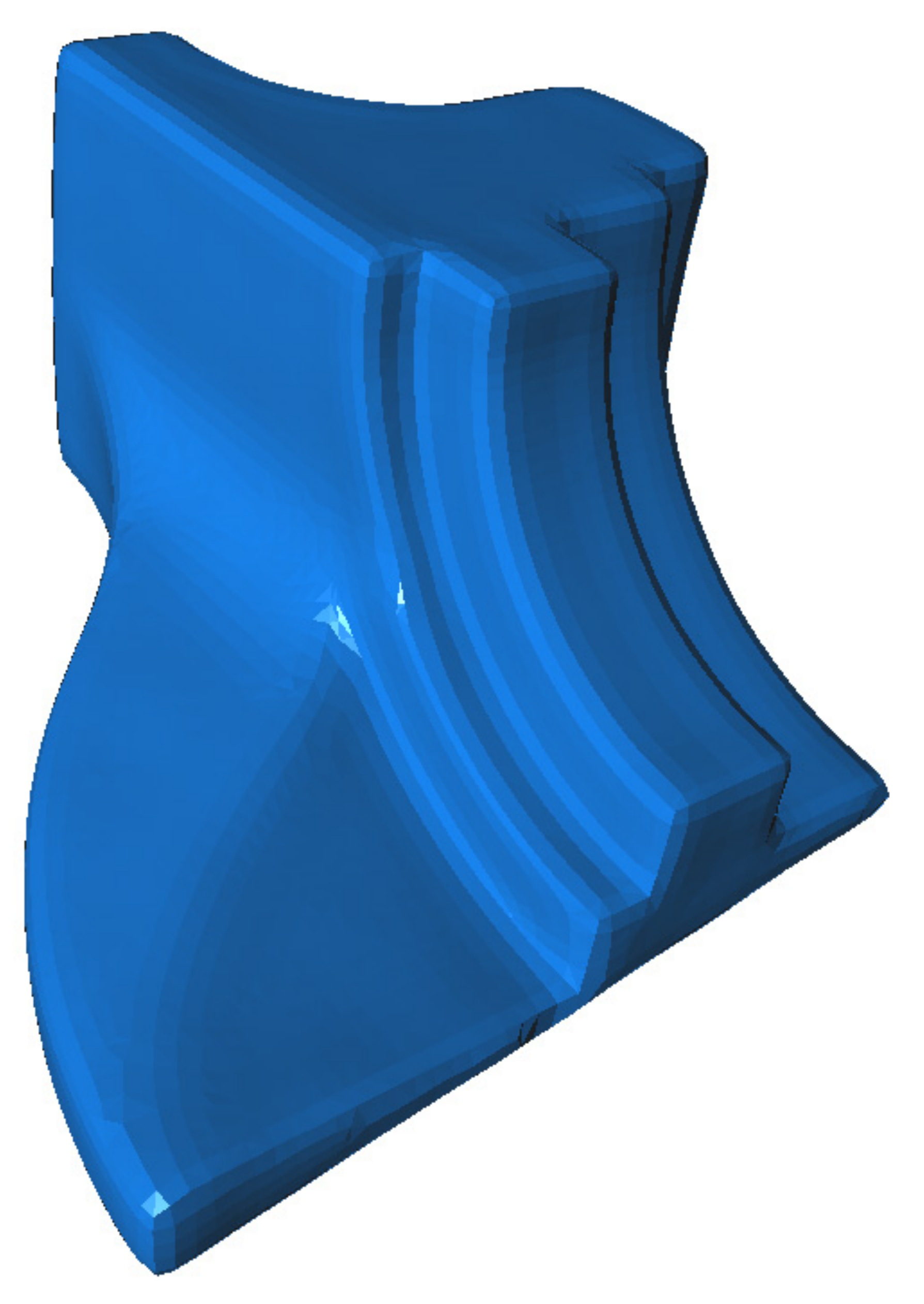}
  \caption{Left: noisy mesh with uniformly distributed noise in normal directions (max. deviation $\pm10\% l_e$);
  middle: 5 iterations by the proposed filter with $\sigma_s=0.05l_e$; right: 5 iterations by the proposed filter with $\sigma_s=0.2l_e$. }
  \label{Fig:Fandisk_denoise_round}
\end{figure}

The two parameters $\sigma_r$ and $\sigma_s$ can be chosen depending on the
vertex density and noise magnitudes of a noisy mesh. When $\sigma_r$ has been
specified, at most $m$ nearest points within a sphere centered at a current
vertex with radius $R=2\sigma_r$ are chosen as the neighborhood of the
vertex. The number $m$ is used to balance the computational costs and
filtering effects. In our experiments we choose $m=100$ by default for high
quality filtering effects and reasonable computational costs. As to the
choice of the parameter $\sigma_s$, a half of the maximum noise magnitude can
be used for feature preserving mesh filtering very well. Mesh noise may not
be filtered when $\sigma_s$ has been set too small a value, but sharp feature
edges or sharp corners will be rounded when $\sigma_s$ has been chosen a
large one. Figure~\ref{Fig:Fandisk_denoise_round} illustrates an example of
mesh filtering with sharp edge preserving or edge and corner rounding by
choosing different values for the parameter $\sigma_s$. The capability of
edge and corner rounding can be used to construct smooth surfaces from rough
initial meshes.

\begin{algorithm}[htb]
\SetAlgoNoLine
\KwIn{Vertex $\mathbf{p}_i$ and its neighborhood $\{(\mathbf{p}_j,\mathbf{n}_j)\}_{j\in N(i)}$. }
\KwOut{New vertex position $\hat{\mathbf{p}}_i$.}
{//} sub\_routine: Compute\_weights\_and\_parameters\;
$s_a=0$; $s_b=0$ \;
\For{each neighbor vertex $\mathbf{p}_j$}
    {
    $c_{ij}=\max\{\mathbf{n}_i^T\mathbf{n}_j,0.001\}$\;
    $d_{ij}=\max\{\frac{1}{2}(|\mathbf{n}_i^T(\mathbf{p}_i-\mathbf{p}_j)|+|\mathbf{n}_j^T(\mathbf{p}_j-\mathbf{p}_i)|),\eta\}$\;
    $\omega_{ij}=\exp(-\frac{d_{ij}^2}{2\sigma_s^2})$\;
    $s_a+=\omega_{ij}d_{ij}$\;
    $s_b+=\omega_{ij}c_{ij}d_{ij}$\;
    }
    $\mu_i=s_a/s_b$\;
{//} main\_routine: Compute\_the\_filtered\_vertex\;
$M_{sum}=0$; $\mathbf{p}_{sum}=0$\;
\For{each neighbor vertex $\mathbf{p}_j$}
    {
    $M_{ij}=\omega_{ij}(I+\mu_i\mathbf{n}_j\mathbf{n}_j^T)$\;
    $M_{sum}+=M_{ij}$\;
    $\mathbf{p}_{sum}+=M_{ij}\mathbf{p}_j$\;
    }
    $M_i^*=\gamma(I-\mathbf{n}_i\mathbf{n}_i^T)$\;
    $M_{sum}+=M_i^*$\;
    $\mathbf{p}_{sum}+=M_i^*\mathbf{p}_i^*$\;
    return $\hat{\mathbf{p}}_i=M_{sum}^{-1}\mathbf{p}_{sum}$\;
\caption{H-MLS Filter For Vertex Filtering} \label{alg:one}
\end{algorithm}

So far we have presented all necessary steps for our proposed filter that
filter one mesh vertex. The algorithm steps are summarized in
Algorithm~\ref{alg:one}. When all vertices of a mesh have been repositioned
by the proposed algorithm, normals at mesh vertices are recomputed for next
iteration of vertex filtering. Usually, a few iterations can lead to
satisfying results.

%-------------------------------------------------------------------------
\section{Results}
\label{Section:results}

Our proposed mesh filtering algorithm was implemented using C++ on a double
2.90Ghz Intel(R) Core(TM) CPU with 8G of RAM. We have applied the filter for
filtering meshes with several different types of noise. Besides meshes with
ordinary triangles, the original mesh can also have highly irregular
triangulation or piecewise flat patches. The parameters and time costs for
the experimental results are summarized in Table~\ref{Table:one}.

\begin{table}
\caption{Parameters and time costs for the examples by the proposed algorithm.}
\label{Table:one}
\begin{minipage}{\columnwidth}
\begin{center} \begin{tabular}{lllllll}
  \toprule
  Model &  Figure & \#Vertex  & $\frac{R}{l_e}$  &  $\frac{\sigma_s}{l_e}$  & \#iter. & Time \\
  \bottomrule
  Dog   &  1(c) & 208k  &   3.6  &  0.4  &  5  & 56s \\
  Bunny &  5    & 35k   &   2    &  0.25 &  10 & 9s  \\
  Face  &  6(e) & 41k   &   2    &  0.08 &  4  & 8s  \\
%  Hand  &  7(e) & 215k  &   2    &  0.1  &  3  & 18s  \\
  Retina&  7(c) & 140k  &   3.6  &  0.6  &  5  & 36s \\
  Wrench&  8(f) & 6k    &   2    &  0.08 &  5  & 2s  \\
%  Crank &  10   & 50k   &   2    &  0.1  &  3  & 7s  \\
  Body  &  9(b)& 45k   &   3.6  &  0.5  &  5  & 12s \\
  \bottomrule
\end{tabular}
\end{center}
%\bigskip\centering
\footnotesize\emph{Note:} The
times do not include the time to load meshes or to compute curvatures.
\end{minipage} \end{table}%

\begin{figure*}[htb]
  \centering
  \subfigure[]{\includegraphics[width=3.1cm]{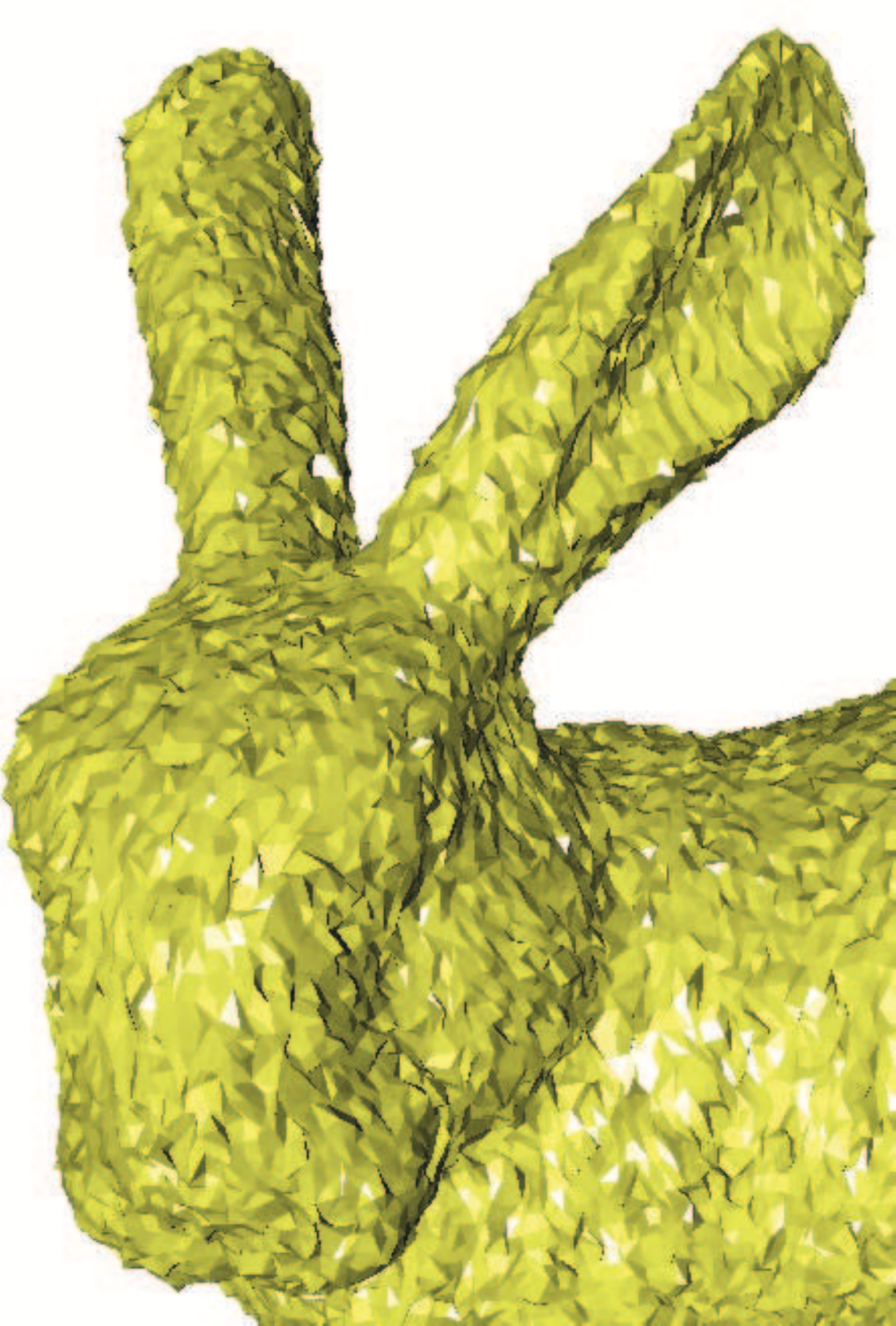}}
  \subfigure[]{\includegraphics[width=3.1cm]{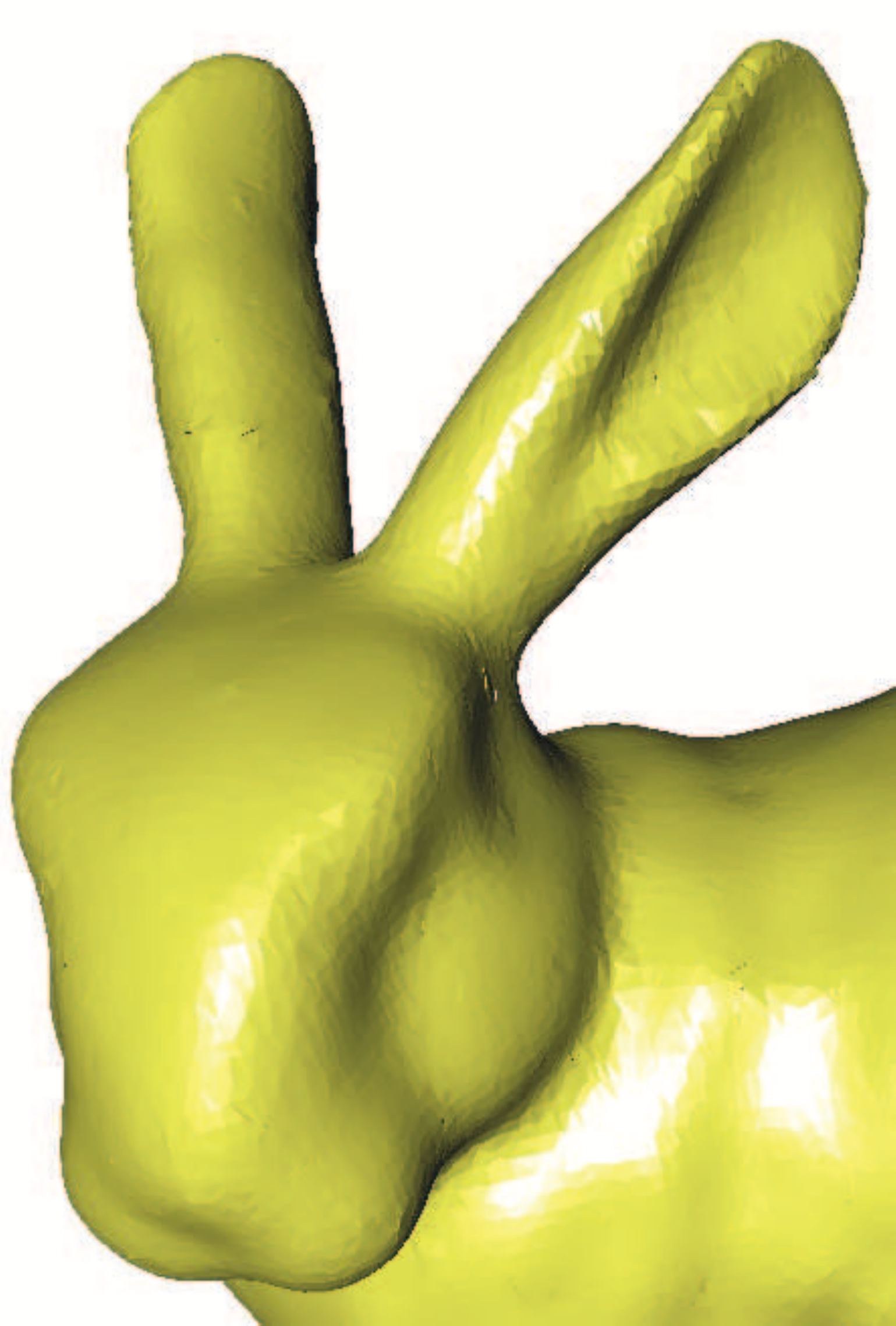}}
  \subfigure[]{\includegraphics[width=3.1cm]{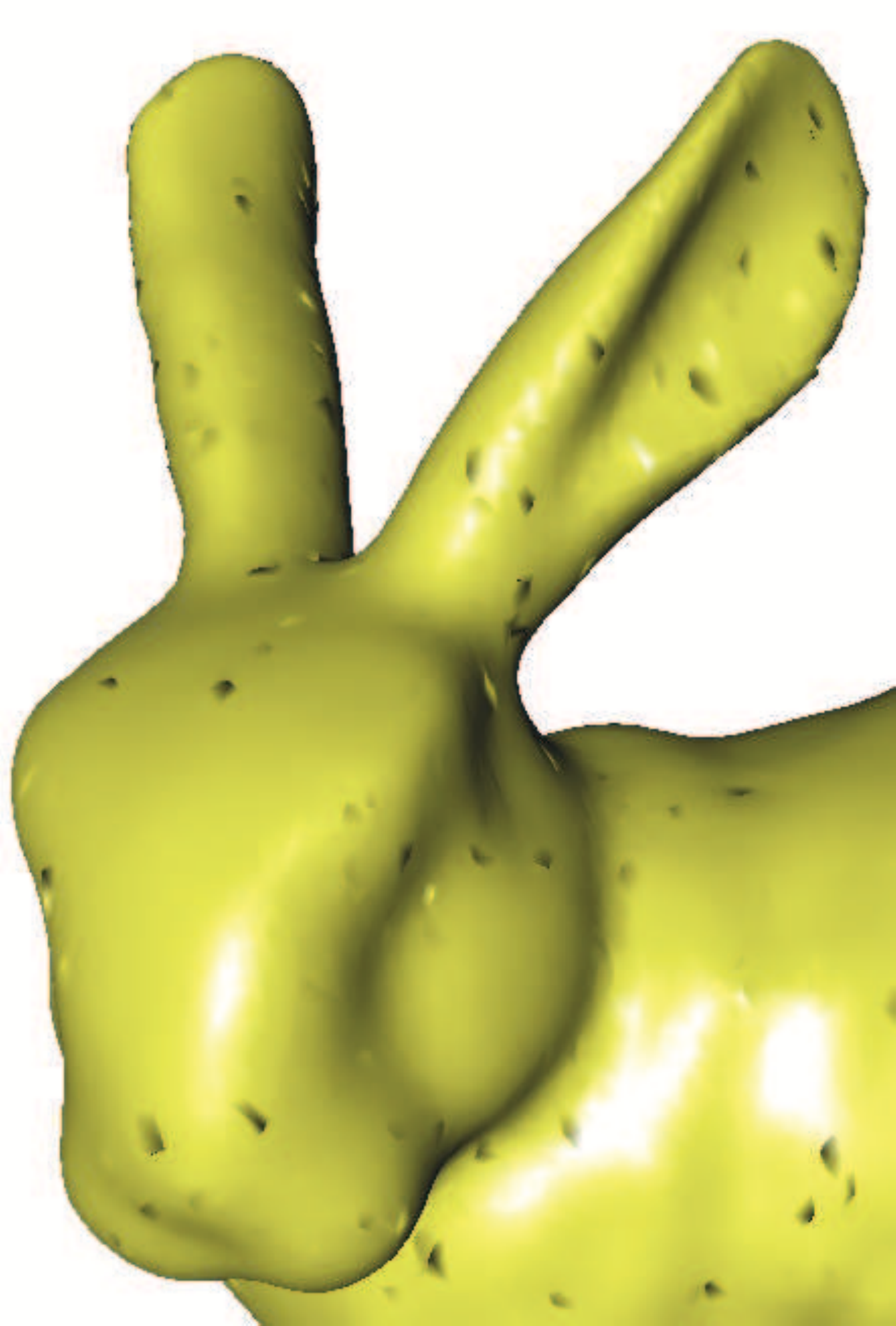}}
  \subfigure[]{\includegraphics[width=3.1cm]{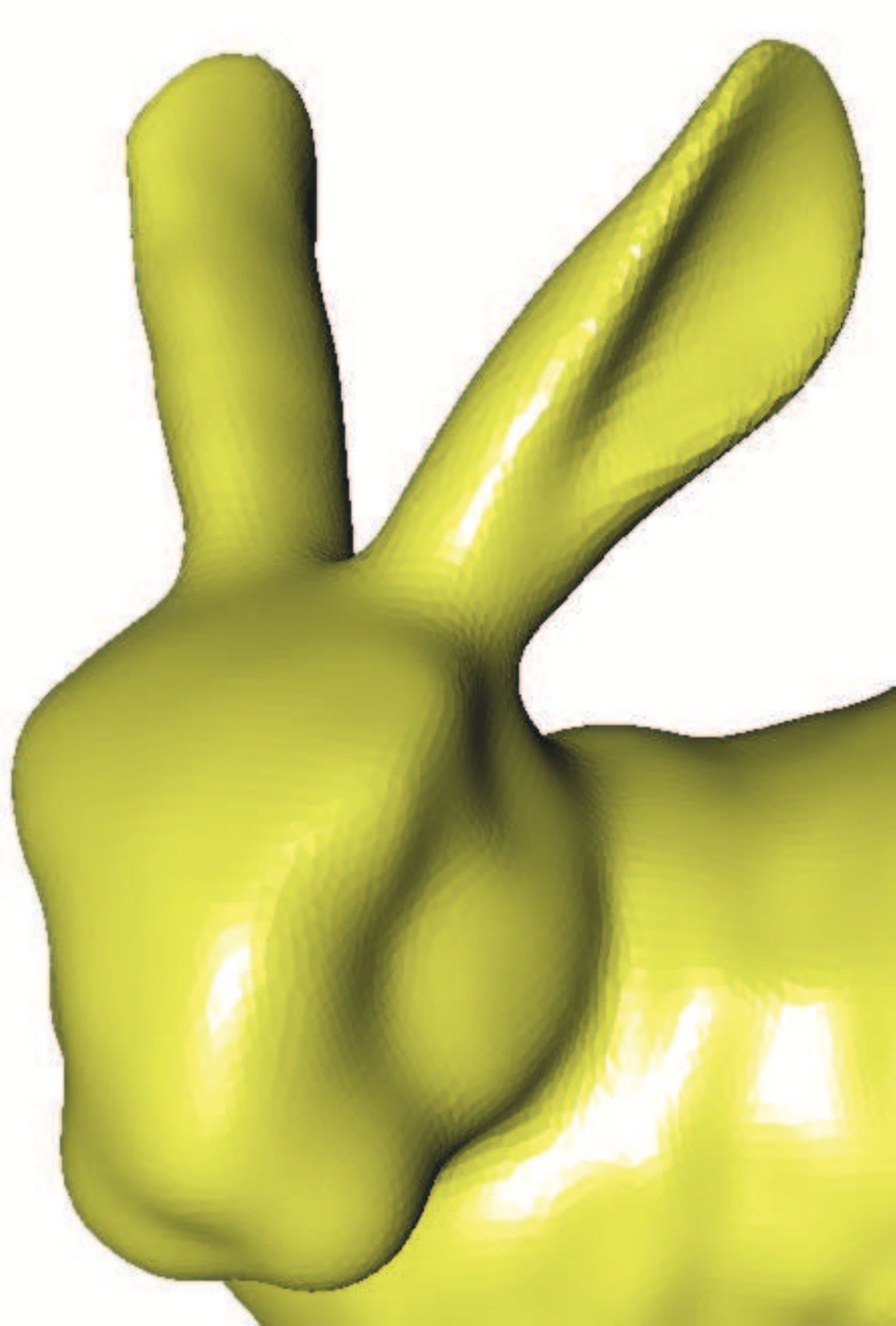}}
  \subfigure[]{\includegraphics[width=3.1cm]{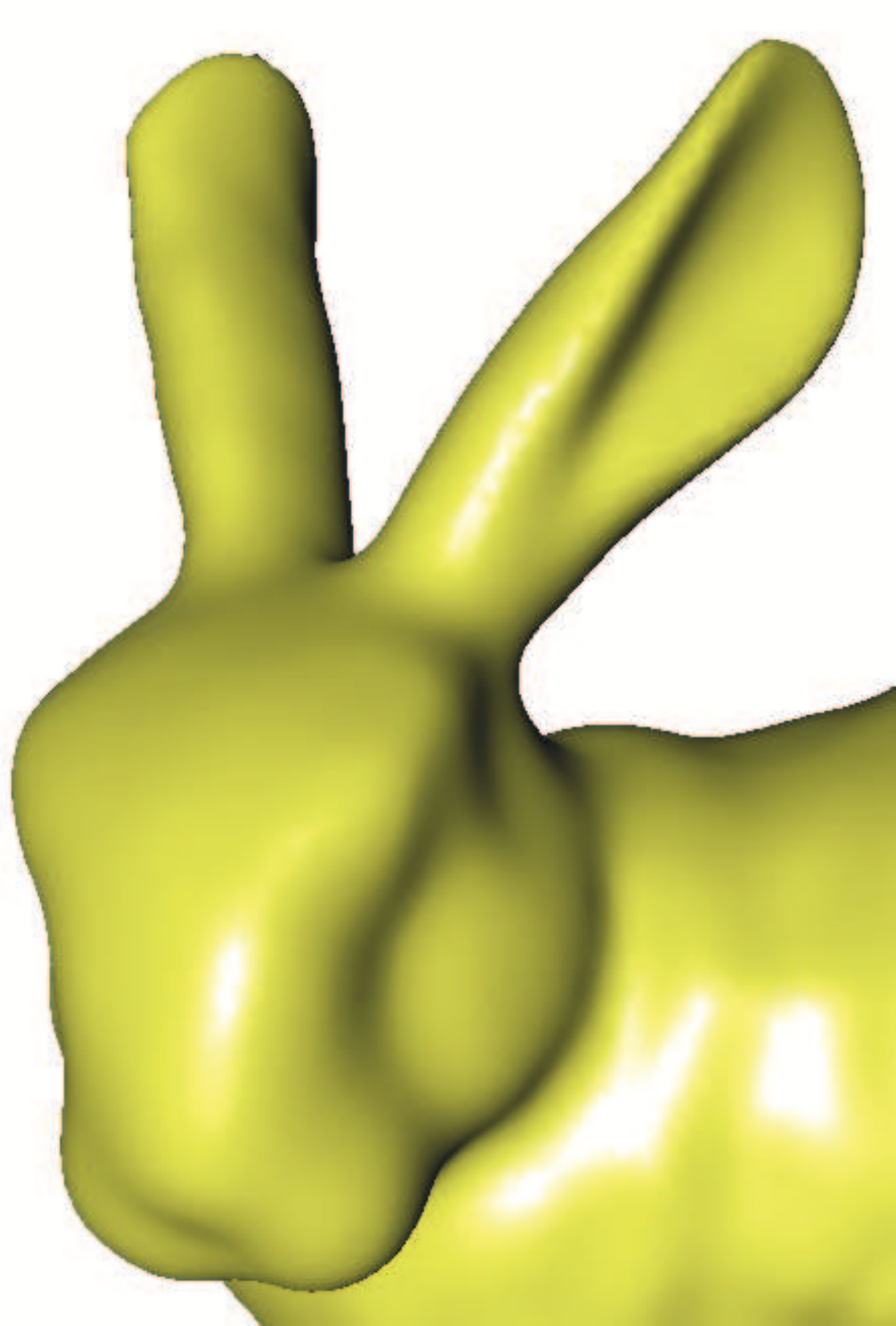}}
  \caption{(a) The mesh corrupted by Gaussian noise in randomly chosen directions ($\sigma=0.5l_e$);
  (b)\&(c) H-MLS($\mathbf{p}_i^*=\mathbf{p}_i$); (d)\&(e)H-MLS($\mathbf{p}_i^*=\mathbf{p}_i^{center}$).
  The filtered meshes in (c) and (e) are Phong shaded to check the flipped edges or folded triangles.}
  \label{Fig:bunny_filtering}
\end{figure*}

Figure~\ref{Fig:bunny_filtering}(a) illustrates a bunny model corrupted with
Gaussian noise in randomly chosen directions. Consequently, the corrupted
mesh has many folded triangles. When the noisy mesh is filtered by our proposed
H-MLS filter under the constraint of lines passing through noisy vertices,
the mesh has been made smooth but folded triangles still exist; see
Figures~\ref{Fig:bunny_filtering}(b)\&(c). If we filter the noisy mesh under the
constraint of lines passing through 1-ring centroids, the obtained surface is
visually smooth and suffers no triangle folding any more; see
Figures~\ref{Fig:bunny_filtering}(d)\&(e) for the result. In the following we
filter vertices under the constraint of lines passing through noisy vertices
or 1-ring centroids according to the criterion that the triangle shapes are to be preserved or more
smooth results are desired.

\begin{figure*}[htb]
  \centering
  \includegraphics[width=3.1cm]{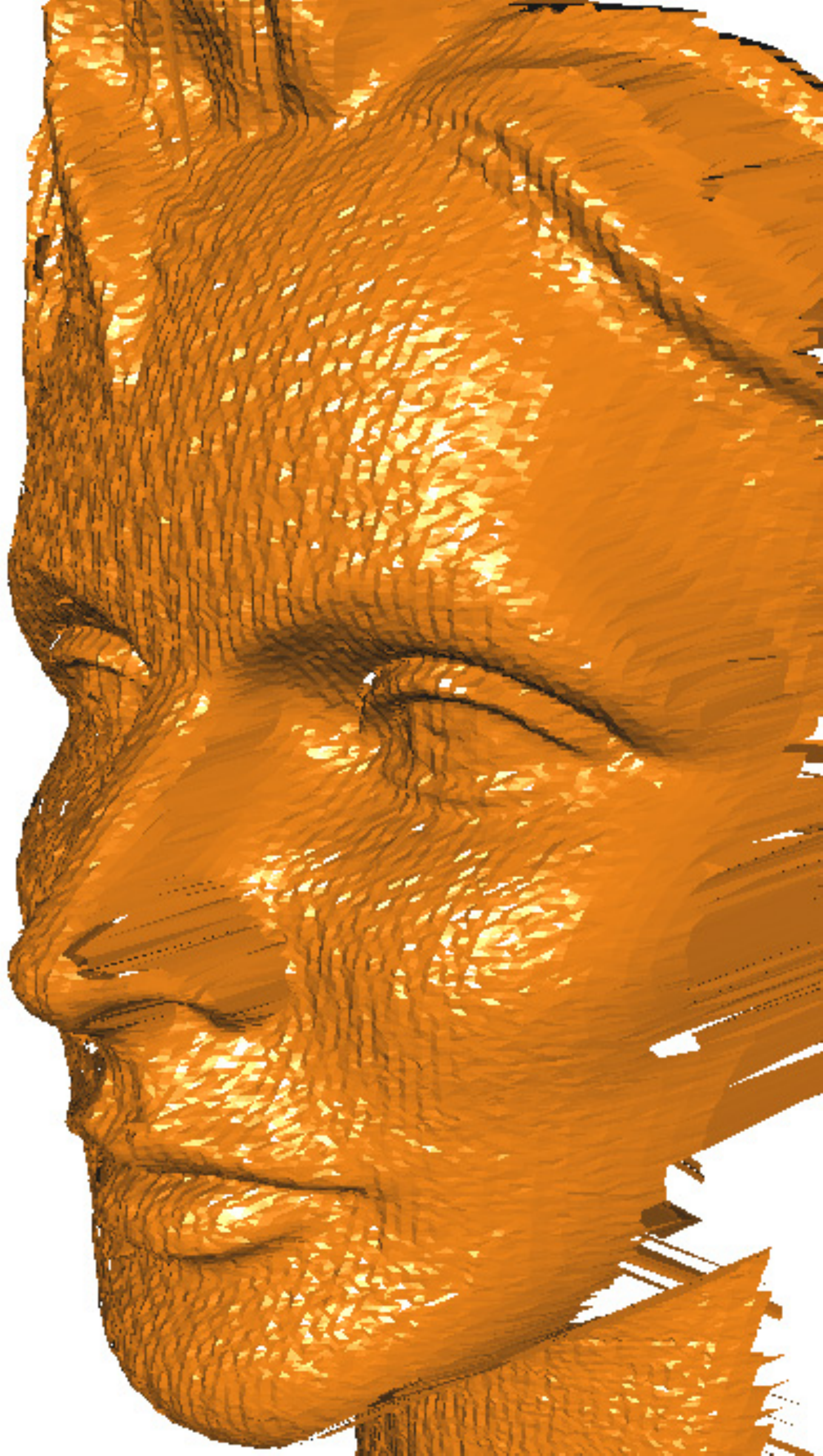}   \hskip 0.cm
  \includegraphics[width=3.1cm]{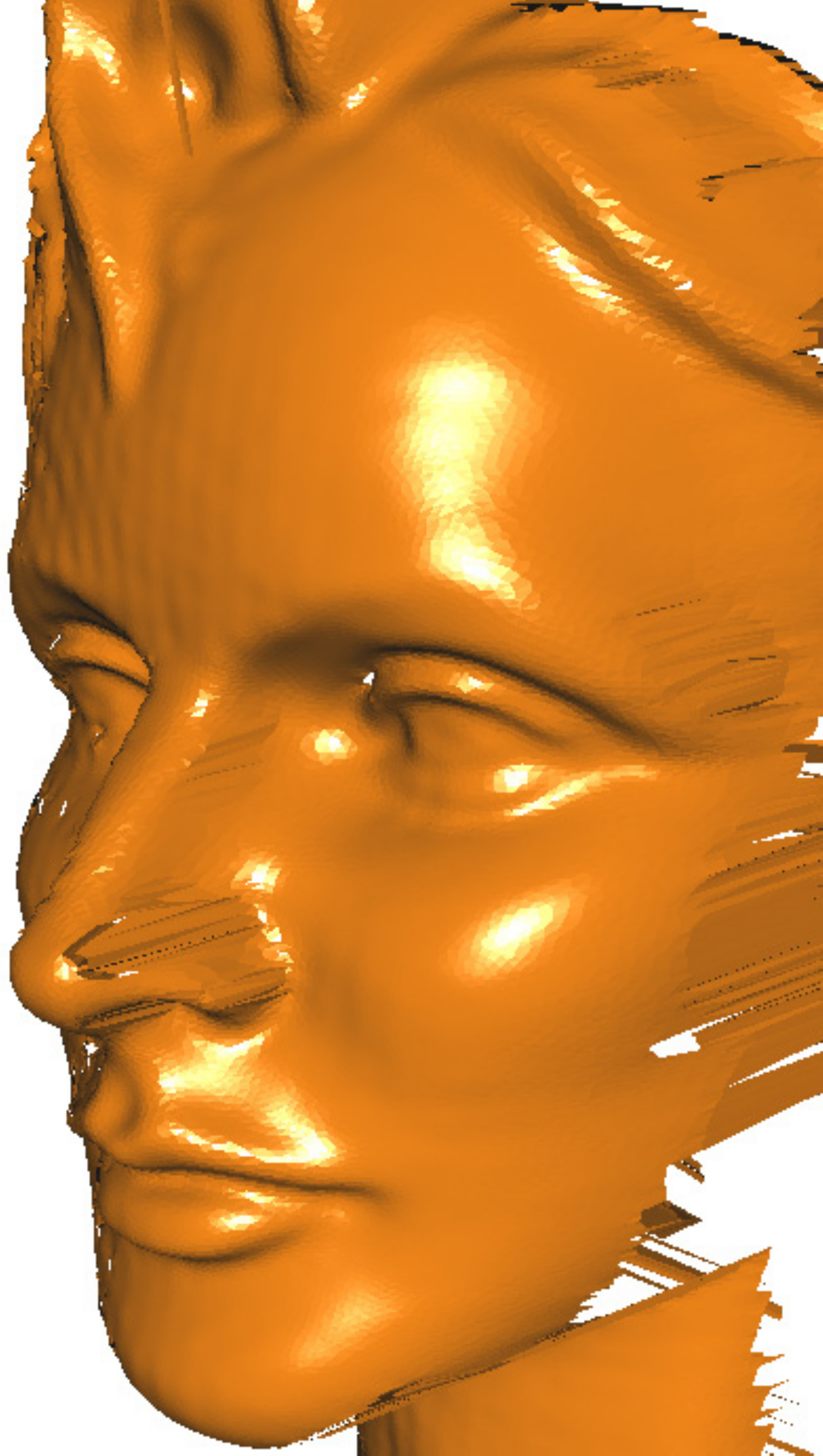}   \hskip 0.cm
  \includegraphics[width=3.1cm]{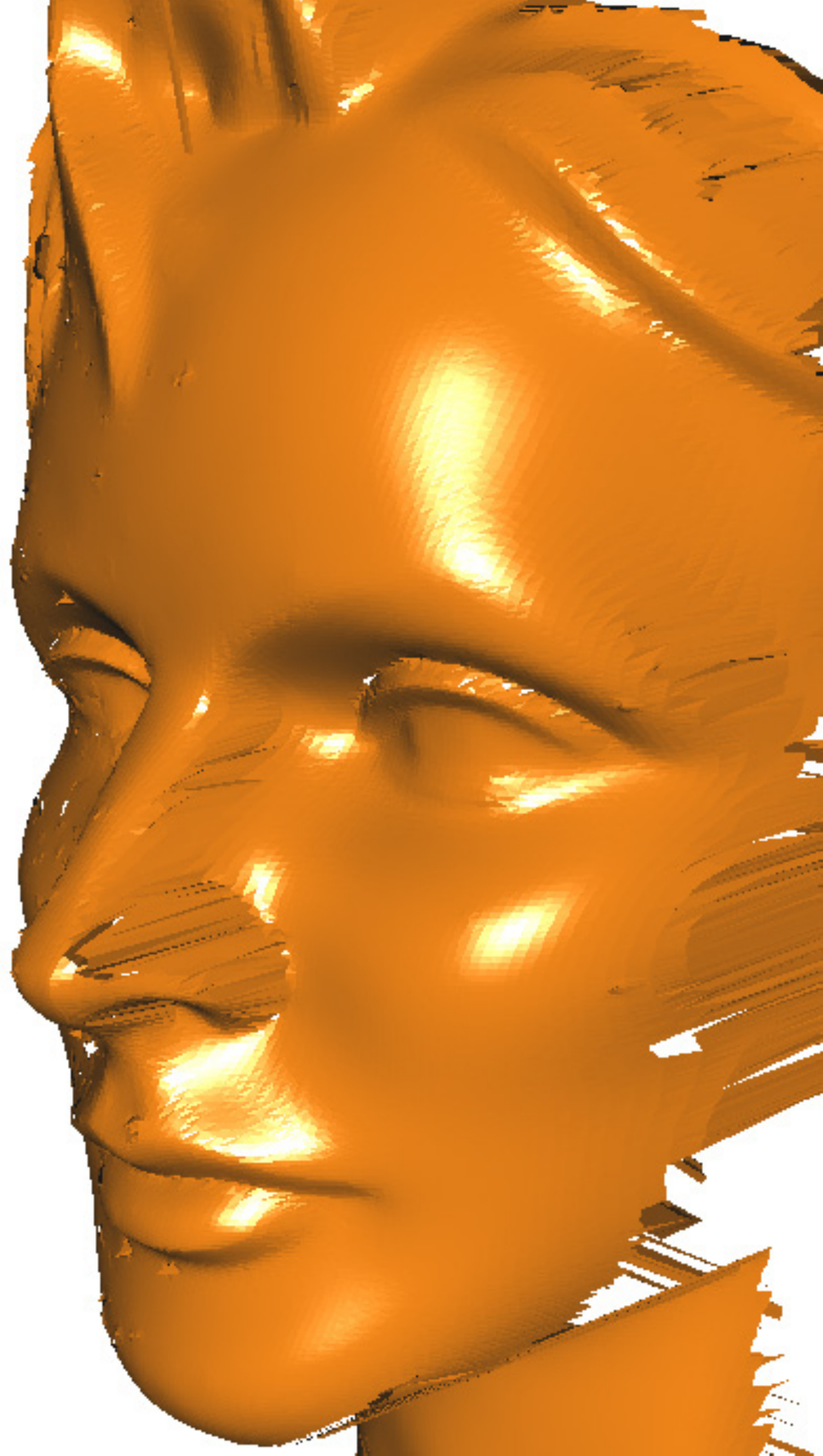}   \hskip 0.cm
  \includegraphics[width=3.1cm]{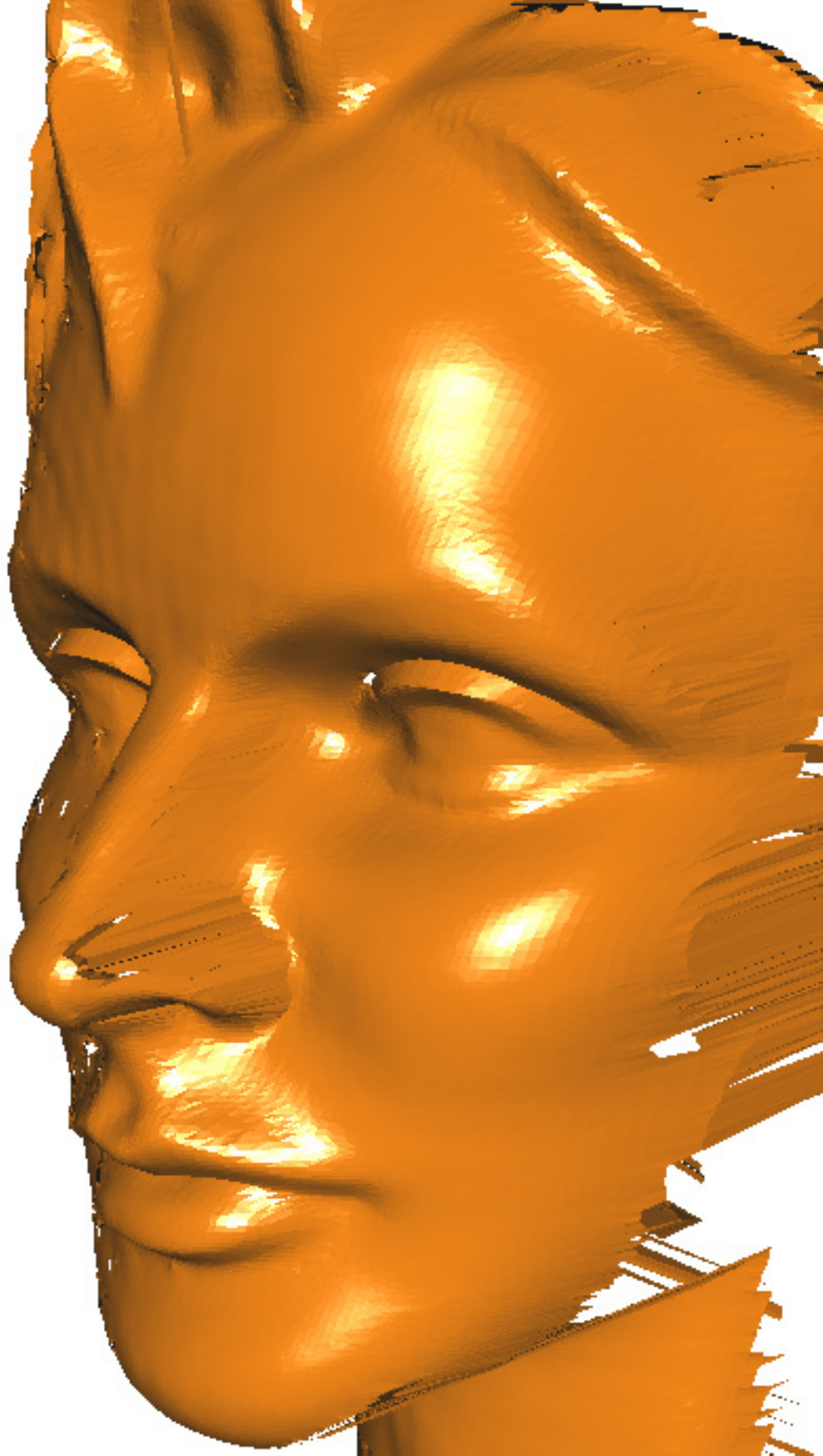}   \hskip 0.cm
  \includegraphics[width=3.1cm]{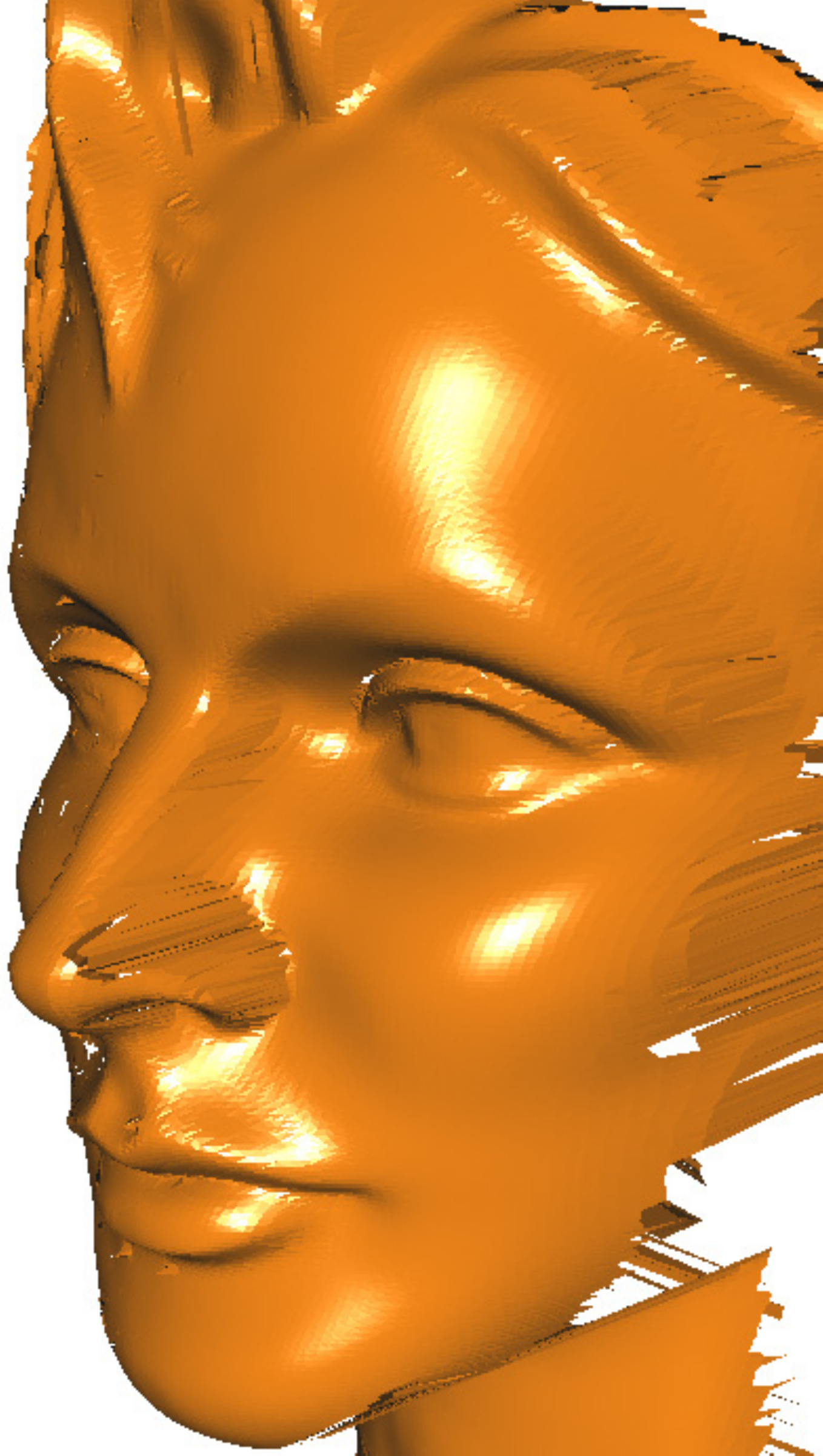}\\
  \includegraphics[width=3.1cm]{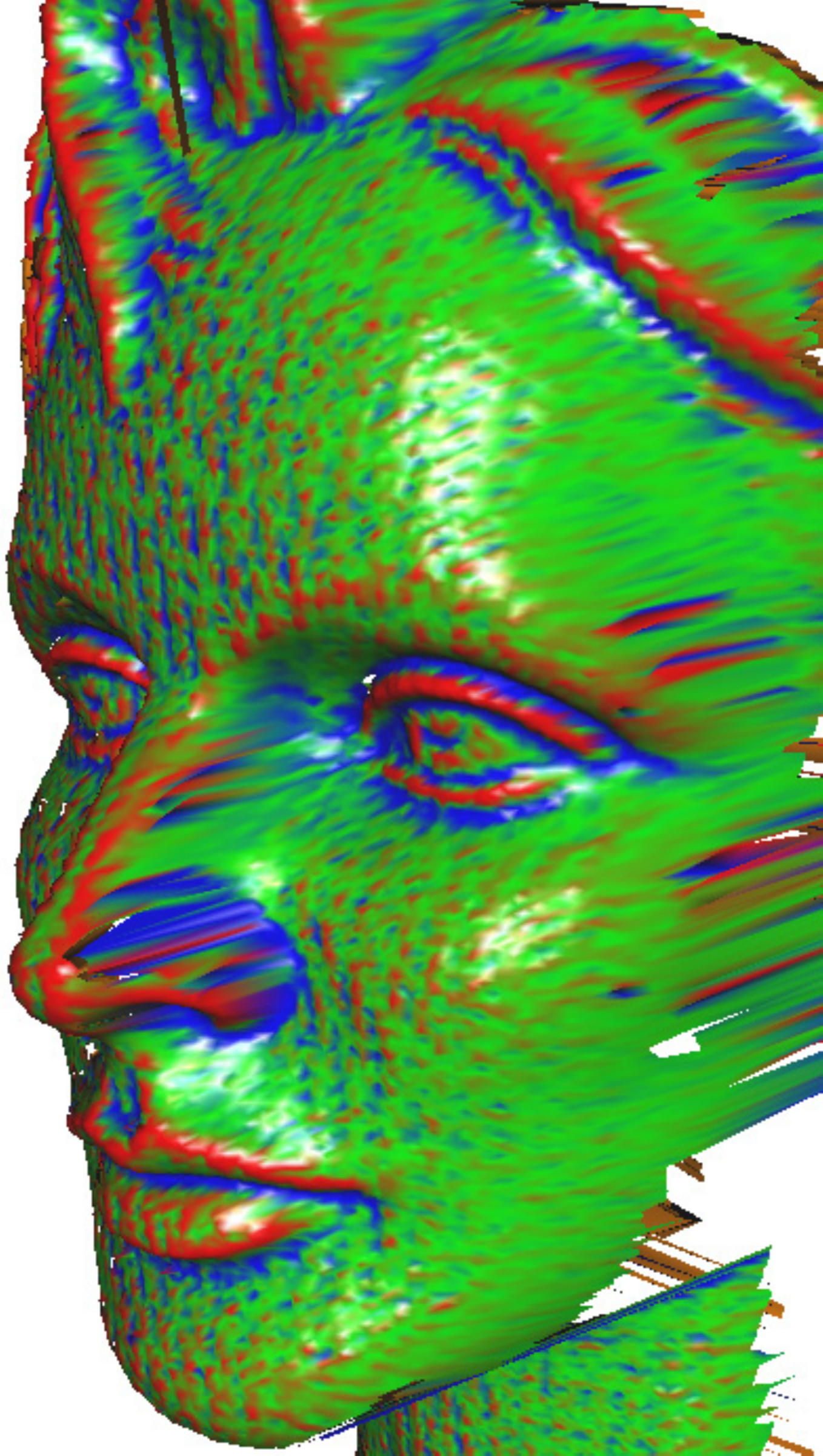}   \hskip 0.cm
  \includegraphics[width=3.1cm]{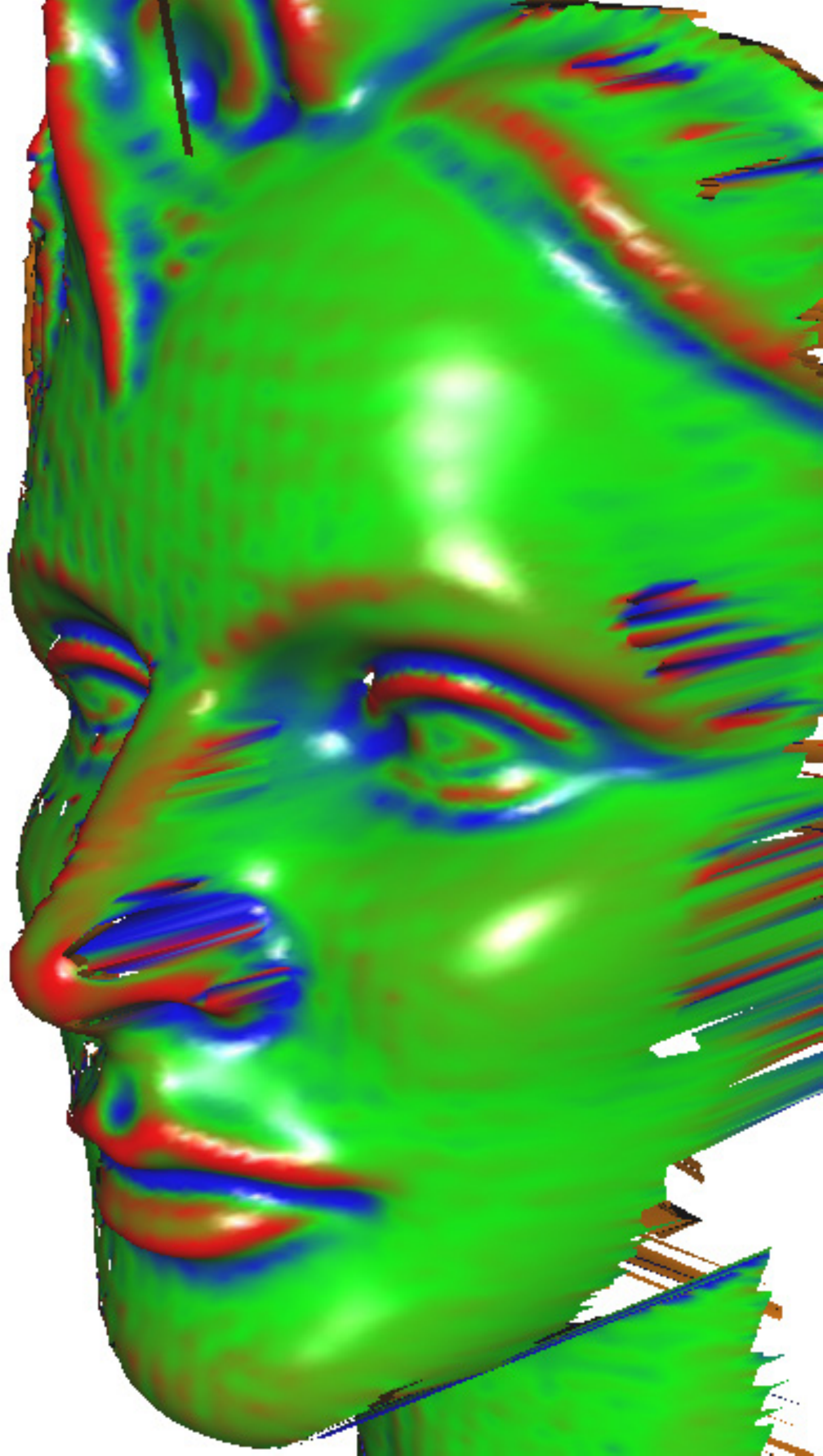}   \hskip 0.cm
  \includegraphics[width=3.1cm]{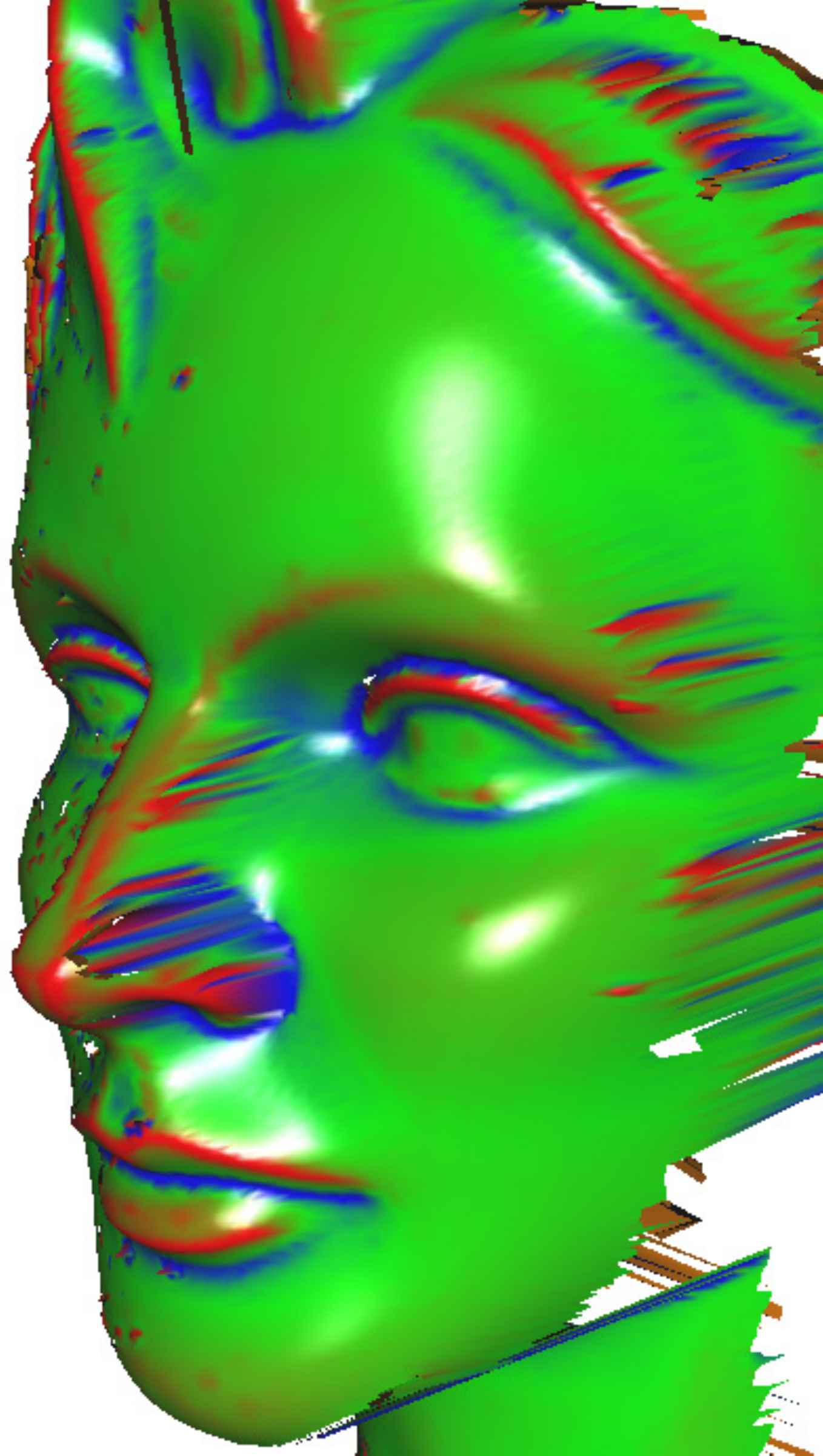}   \hskip 0.cm
  \includegraphics[width=3.1cm]{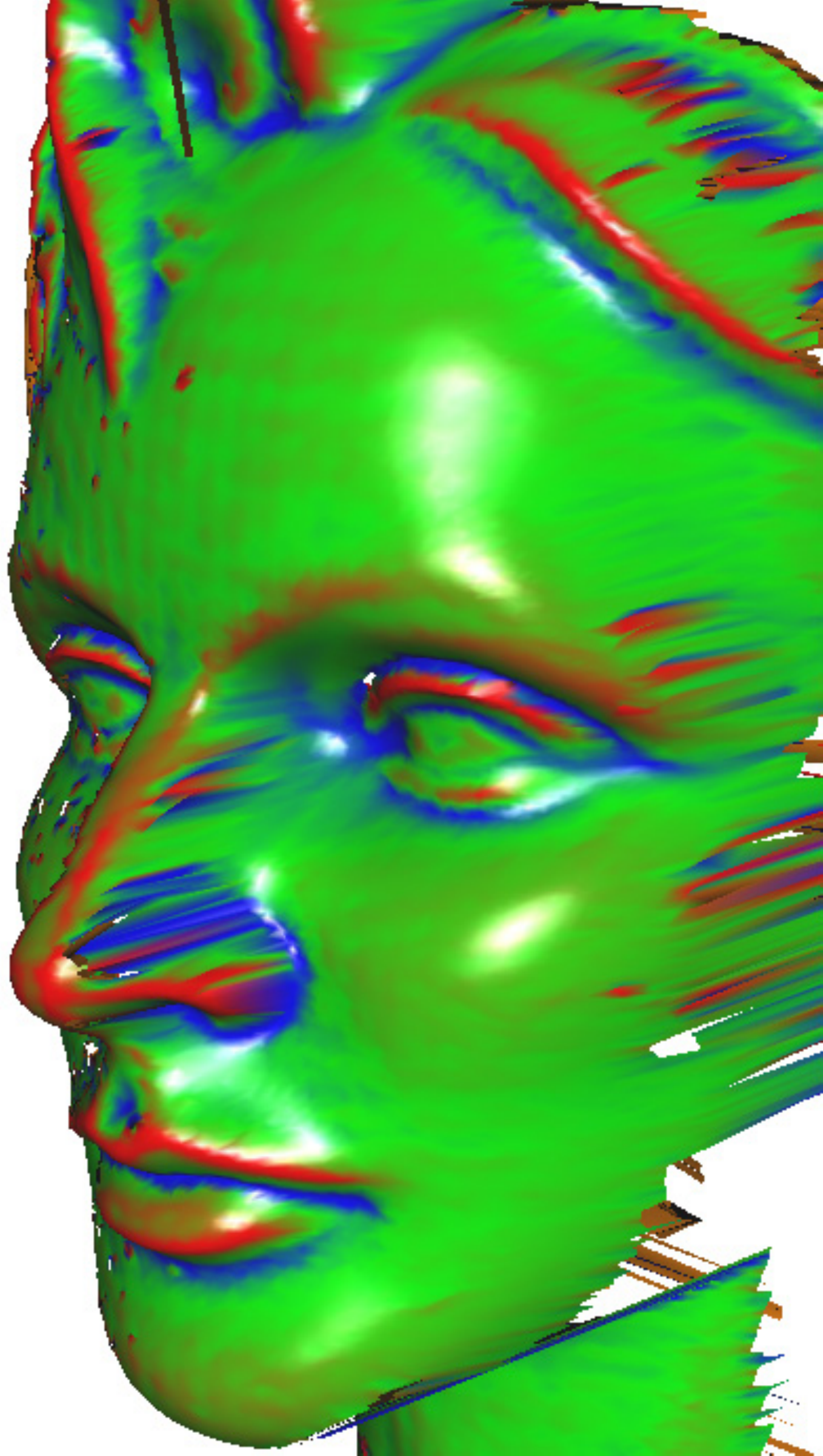}   \hskip 0.cm
  \includegraphics[width=3.1cm]{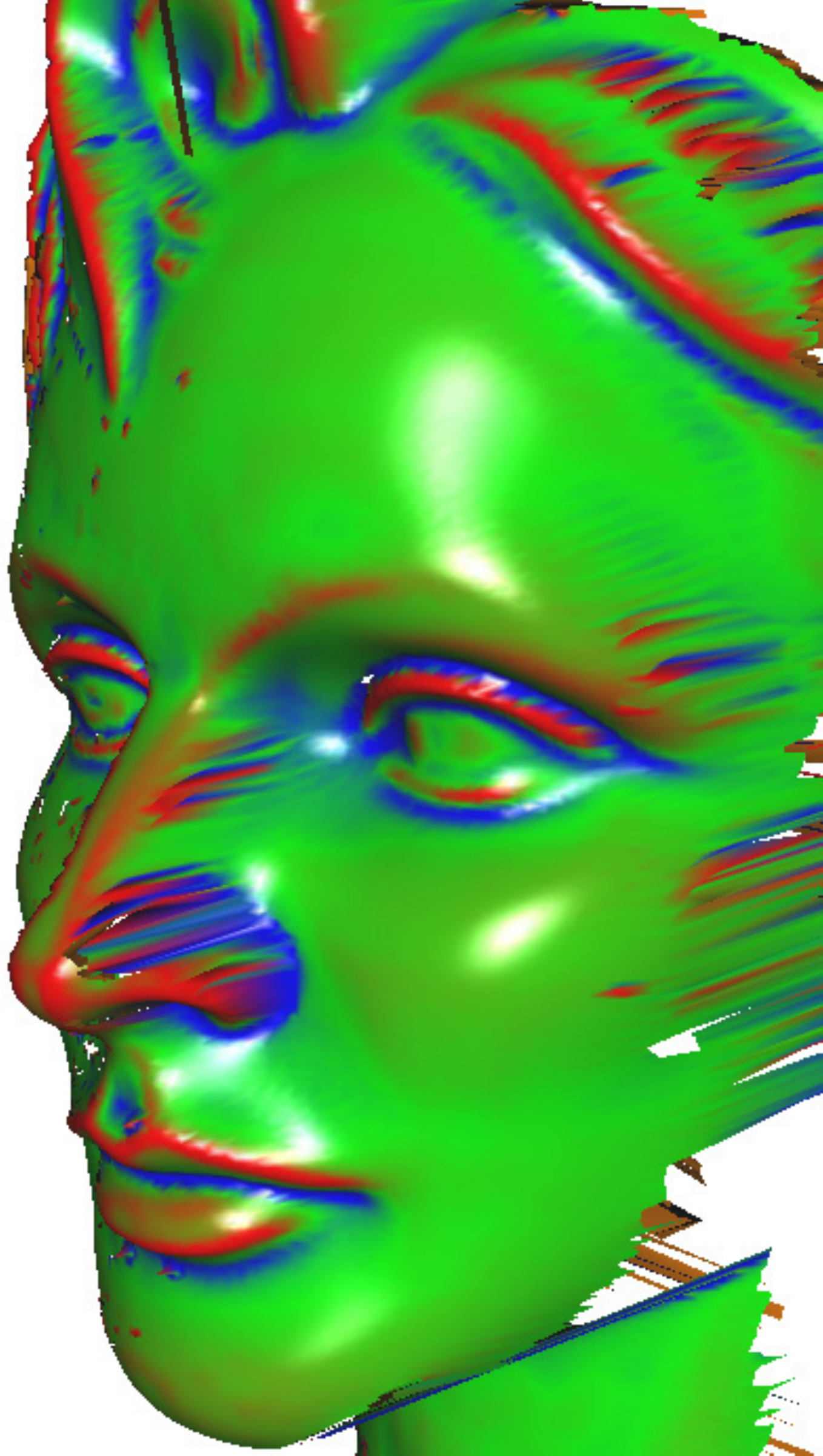}\\
  \vskip -0.005cm
  \subfigure[]{\includegraphics[width=3.0cm]{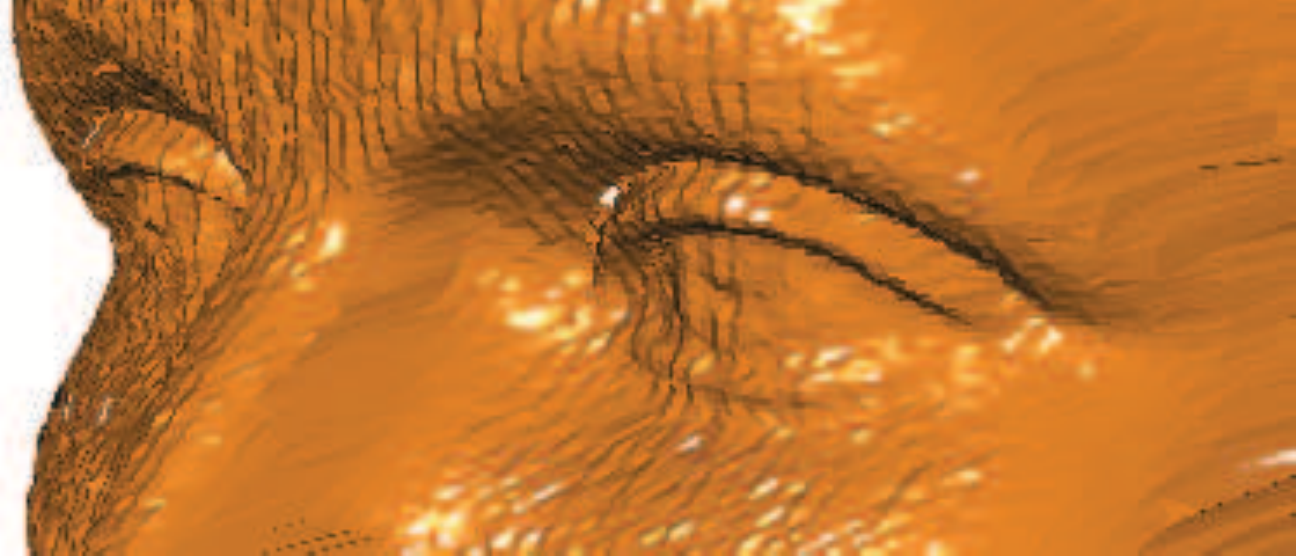}}   \hskip 0.cm
  \subfigure[]{\includegraphics[width=3.0cm]{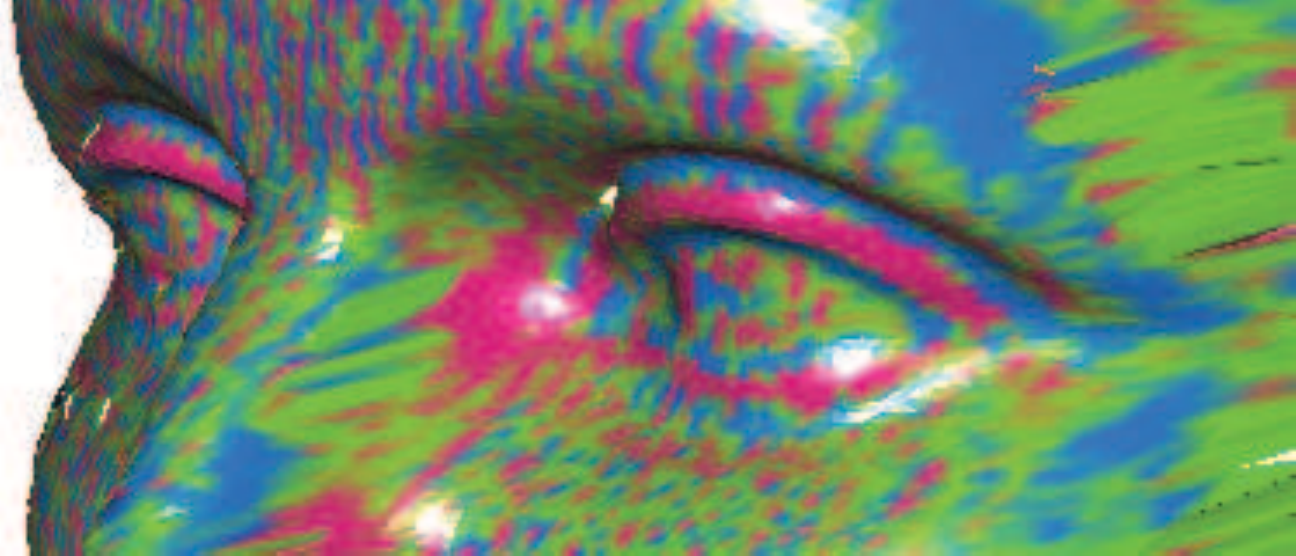}}   \hskip 0.cm
  \subfigure[]{\includegraphics[width=3.0cm]{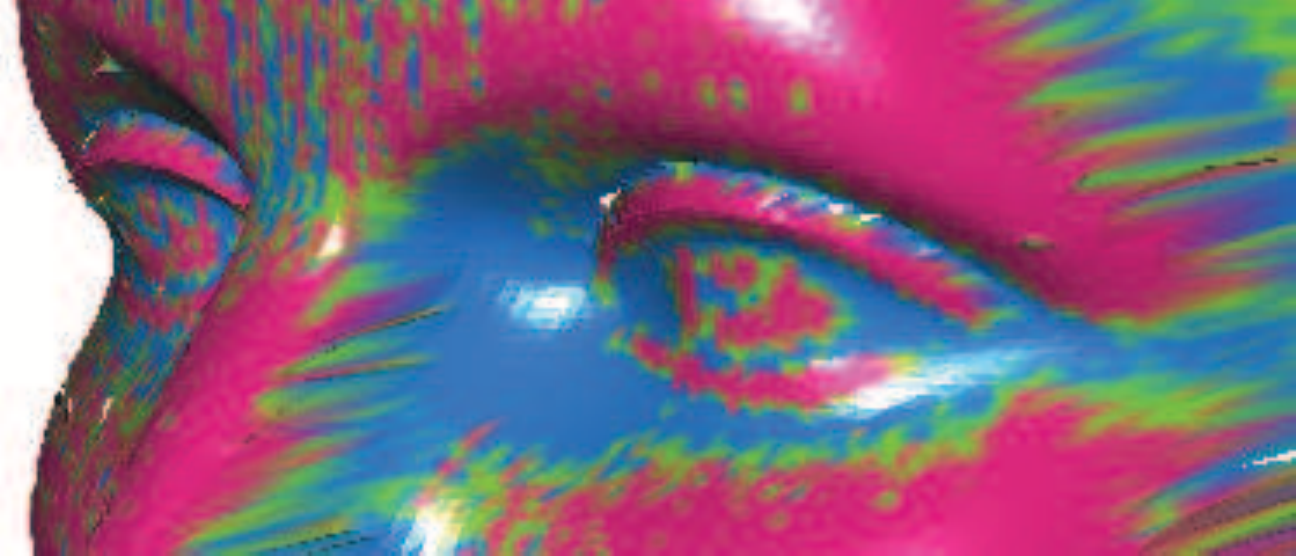}}   \hskip 0.cm
  \subfigure[]{\includegraphics[width=3.0cm]{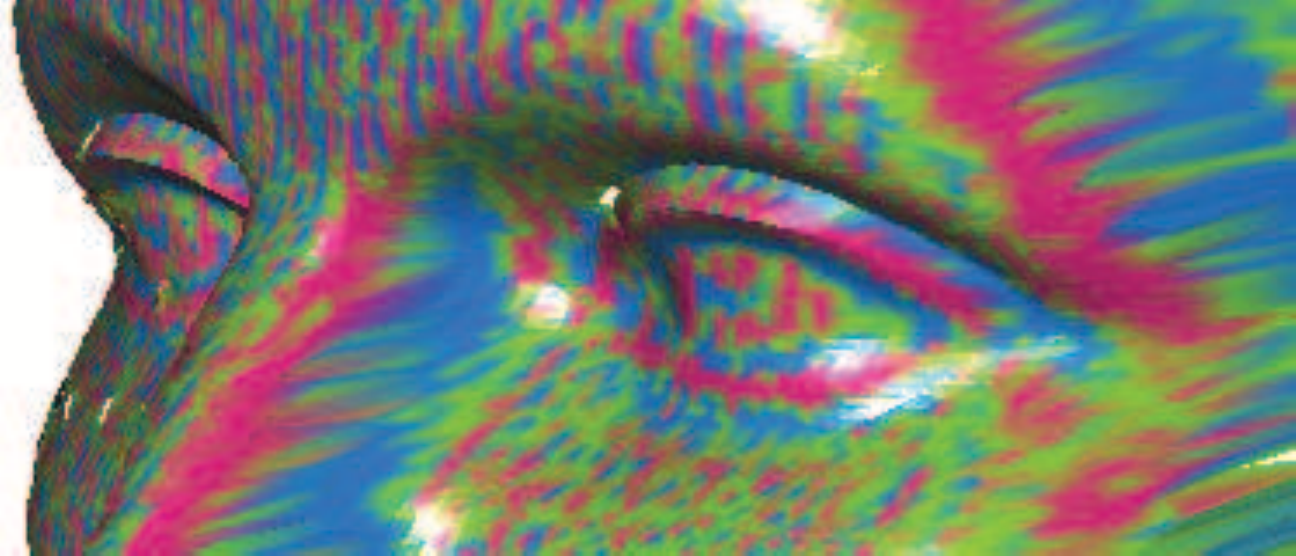}}   \hskip 0.cm
  \subfigure[]{\includegraphics[width=3.0cm]{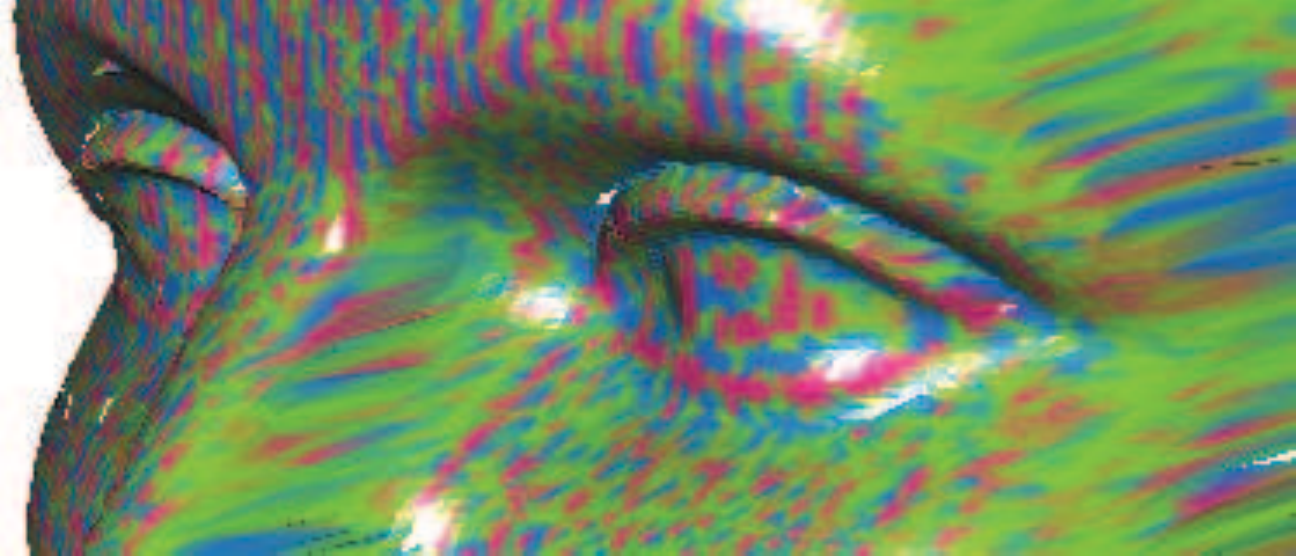}}
  \caption{(a) The input noisy mesh; (b) PMC; (c) BMF; (d) BNF; (e) proposed($\mathbf{p}_i^\ast=\mathbf{p}_i$).
  Middle row: mean curvature plots, where the red color represents the highest value while the blue color means the lowest.
  Bottom row: error plots for a selected region.}
  \label{Fig:femmeN_denoise}
\end{figure*}

%\begin{figure*}[htb]
%  \centering
%  \includegraphics[width=3.1cm]{rawhand_orig.pdf}   \hskip 0.cm
%  \includegraphics[width=3.1cm]{rawhand_pmc20.pdf}   \hskip 0.cm
%  \includegraphics[width=3.1cm]{rawhand_bmf3.pdf}   \hskip 0.cm
%  \includegraphics[width=3.1cm]{rawhand_bnf10.pdf}   \hskip 0.cm
%  \includegraphics[width=3.1cm]{rawhand_pnf3.pdf}\\
%  \includegraphics[width=3.1cm]{rawhand_orig_km.pdf}   \hskip 0.cm
%  \includegraphics[width=3.1cm]{rawhand_pmc20_km.pdf}   \hskip 0.cm
%  \includegraphics[width=3.1cm]{rawhand_bmf3_km.pdf}   \hskip 0.cm
%  \includegraphics[width=3.1cm]{rawhand_bnf10_km.pdf}   \hskip 0.cm
%  \includegraphics[width=3.1cm]{rawhand_pnf3_km.eps}\\
%  \subfigure[]{\includegraphics[width=3.1cm]{rawhand_orig_err.pdf}}   \hskip 0.cm
%  \subfigure[]{\includegraphics[width=3.1cm]{rawhand_pmc20_err.pdf}}   \hskip 0.cm
%  \subfigure[]{\includegraphics[width=3.1cm]{rawhand_bmf3_err.pdf}}   \hskip 0.cm
%  \subfigure[]{\includegraphics[width=3.1cm]{rawhand_bnf10_err.pdf}}   \hskip 0.cm
%  \subfigure[]{\includegraphics[width=3.1cm]{rawhand_pnf3_err.pdf}}
%  \caption{(a) The input noisy mesh; (b) PMC; (c) BMF; (d) BNF; (e) proposed($\mathbf{p}_i^\ast=\mathbf{p}_i^{center}$).
%  Middle row: mean curvature plots. Bottom row: error plots of a selected region. }
%  \label{Fig:rawhand_denoise}
%\end{figure*}

Figure~\ref{Fig:femmeN_denoise} shows a real scan-reconstructed triangular
mesh of a face. We have compared our filtering result with those by prescribed
mean curvature (PMC) flow~\citep{Polthier04}, bilateral mesh denoising
(BMF)~\citep{Fleishman03} or bilateral normal filtering
(BNF)~\citep{ZhengFuTai2011}. Similar to ours, these three methods are
conceptually simple and can preserve salient features with no complex
pre-computation or costly optimization steps. Particularly, the parameters
for each known method are chosen carefully to achieve best filtering results.
To visualize the surface features and the filtering results clearly, the mean
curvature plots of the meshes before or after filtering are computed by
employing the method in~\citep{Goldfeather2004}. From the figure we see
that both PMC and BNF can preserve or even sharpen local features (e.g.
eyelids), but the obtained meshes still suffer low frequency noises in
non-feature regions. The BMF can give a smooth filtering result in
non-feature regions but minor features have been over smoothed due to the
well known shrinkage property of the method. Our proposed method can overcome
these shortcomings effectively, a high fidelity filtered mesh with neither
over-smoothed areas nor sharpened features is obtained.

%Figure~\ref{Fig:rawhand_denoise} illustrates another example of real scan-reconstructed surface filtering. We filter the raw hand model by 3 iterations of H-LMS filtering with $\mathbf{p}_i=\mathbf{p}_i^{center}$ for each vertex refinement. From the figure we can see that a visually smooth surface is obtained after filtering. The salient high curvature regions within original surface are preserved very well and the filtered surface has a high fidelity to the original surface. As a comparison, we also filter the original noisy mesh by PMC, BMF or BNF, all with parameters to achieve the best results. All these methods can give visually smooth results, but the filtered surfaces either suffer smoothness artifacts or deformation in high curvature regions.

\begin{figure*}[htb]
  \centering
  \subfigure[]{\includegraphics[width=5.2cm]{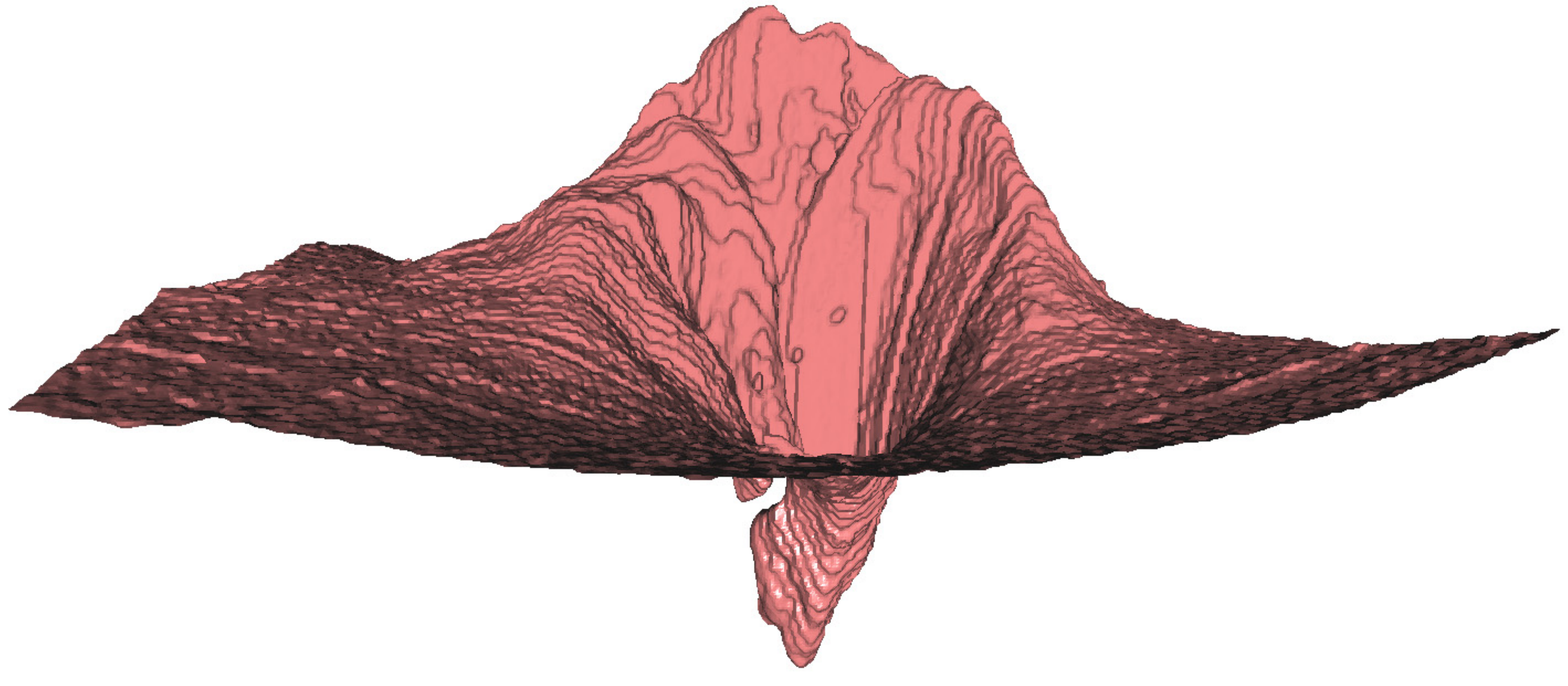}}
  \subfigure[]{\includegraphics[width=5.2cm]{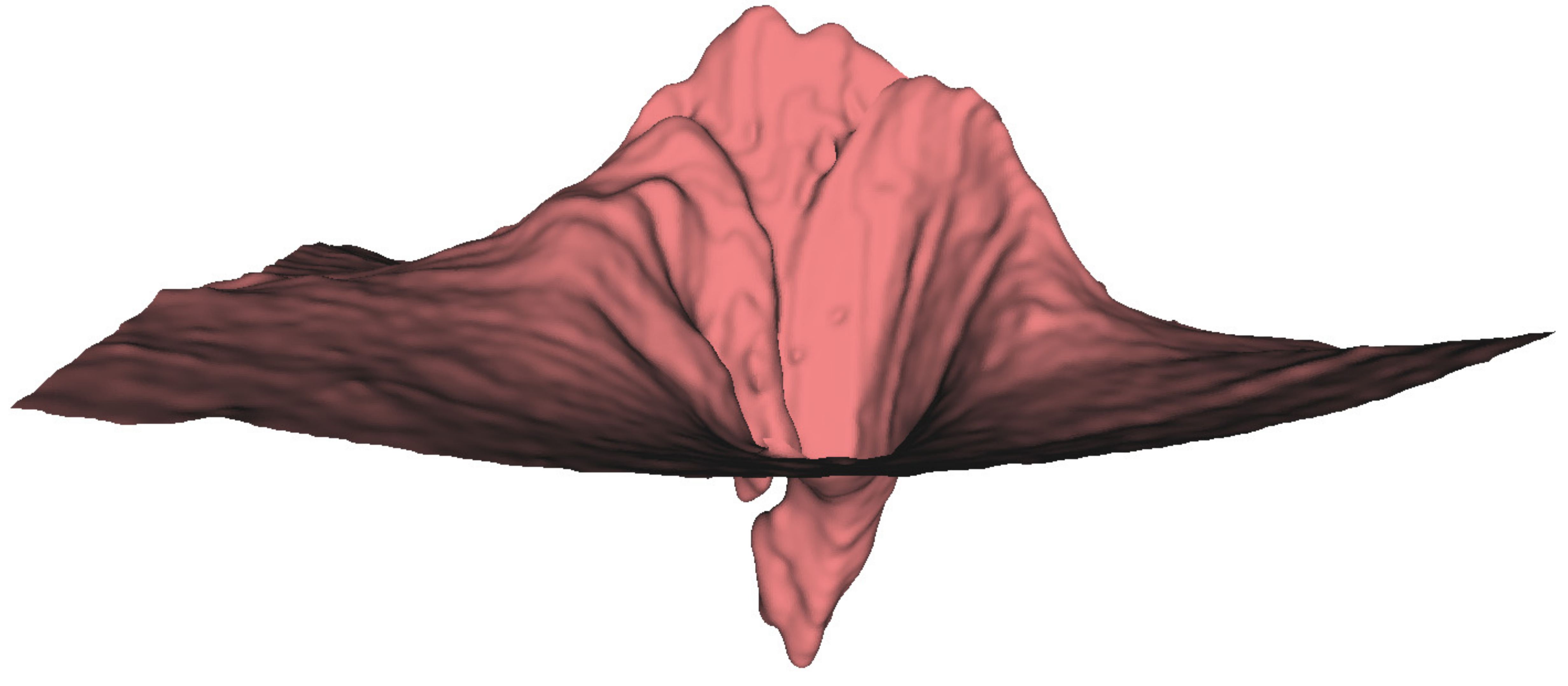}}
  \subfigure[]{\includegraphics[width=5.2cm]{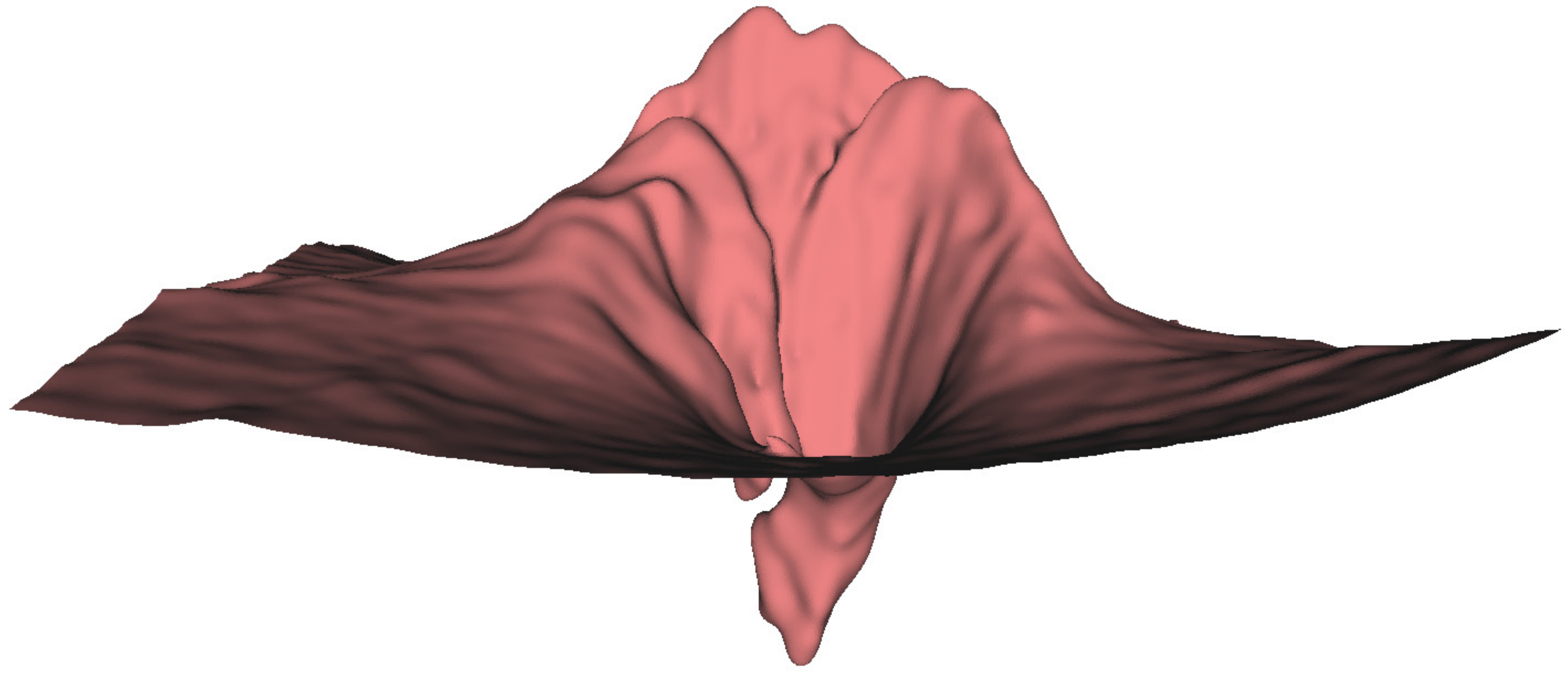}}
  \caption{(a) The original mesh; (b) high fidelity mesh denoising~\citep{YadavRP18:HighFidelityDenoising}
  (the original and the filtered meshes courtesy of S.K.Yadav); (c) proposed
($\mathbf{p}_i^\ast=\mathbf{p}_i^{center}$).} \label{Fig:retina_smoothing}
\end{figure*}

Figure~\ref{Fig:retina_smoothing} shows a mesh of a human retina constructed
from medical image data. The original surface has many undesired stairs due
to limited precision of the marching cubes algorithm. Most noises and minor
stairs have been successfully removed by a high fidelity denoising algorithm
proposed by Yadav et al.~(\citeyear{YadavRP18:HighFidelityDenoising}), but stairs
at large steps are still visible. Instead of using hard position constraint,
we filter the noisy mesh by our proposed H-LMS filter with weights and
parameters adaptively computed from the noisy data. By choosing properly
sized neighborhood for mesh vertex filtering, a visually smooth surface with
neither minor step stairs nor large ones is obtained.

\begin{figure*}[htb]
  \centering
%  \subfigure[]{\includegraphics[width=2.6cm]{BoneWrench_truth.pdf}}
  \subfigure[]{\includegraphics[width=3.0cm]{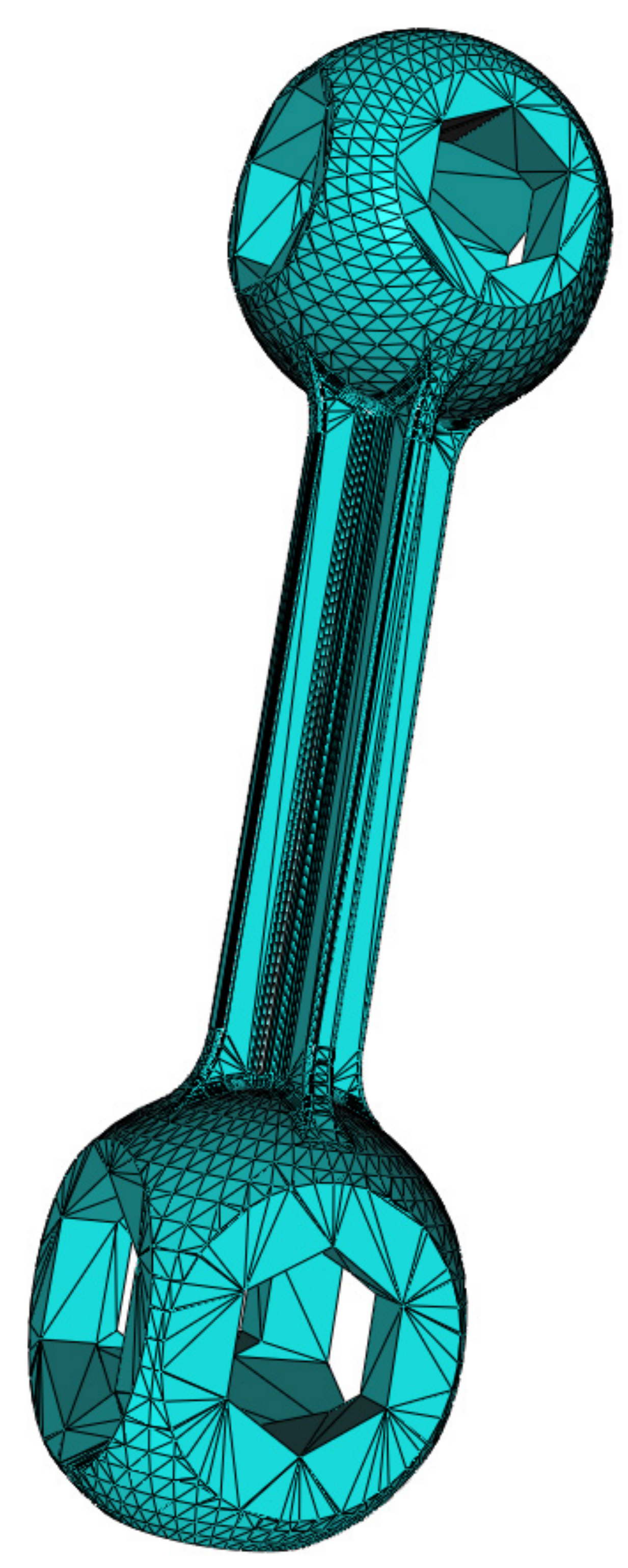}}
  \subfigure[]{\includegraphics[width=3.0cm]{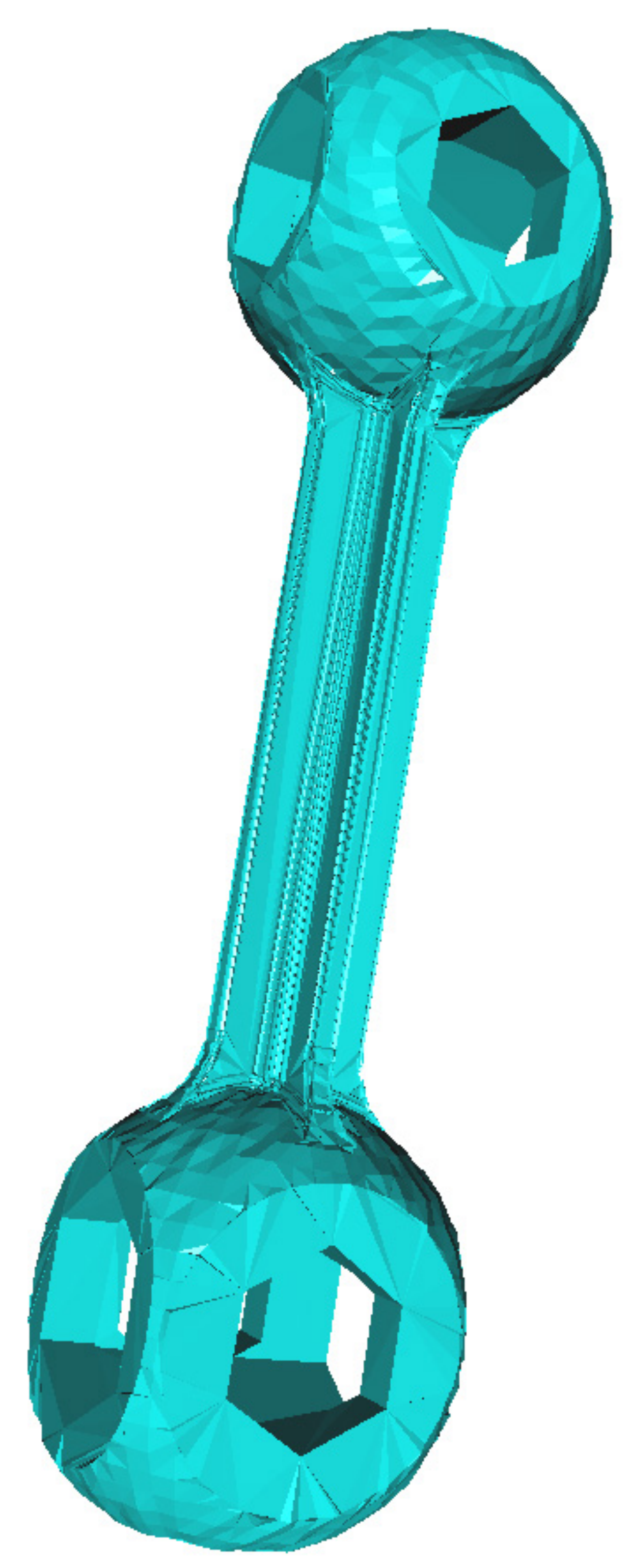}}
  \subfigure[]{\includegraphics[width=3.0cm]{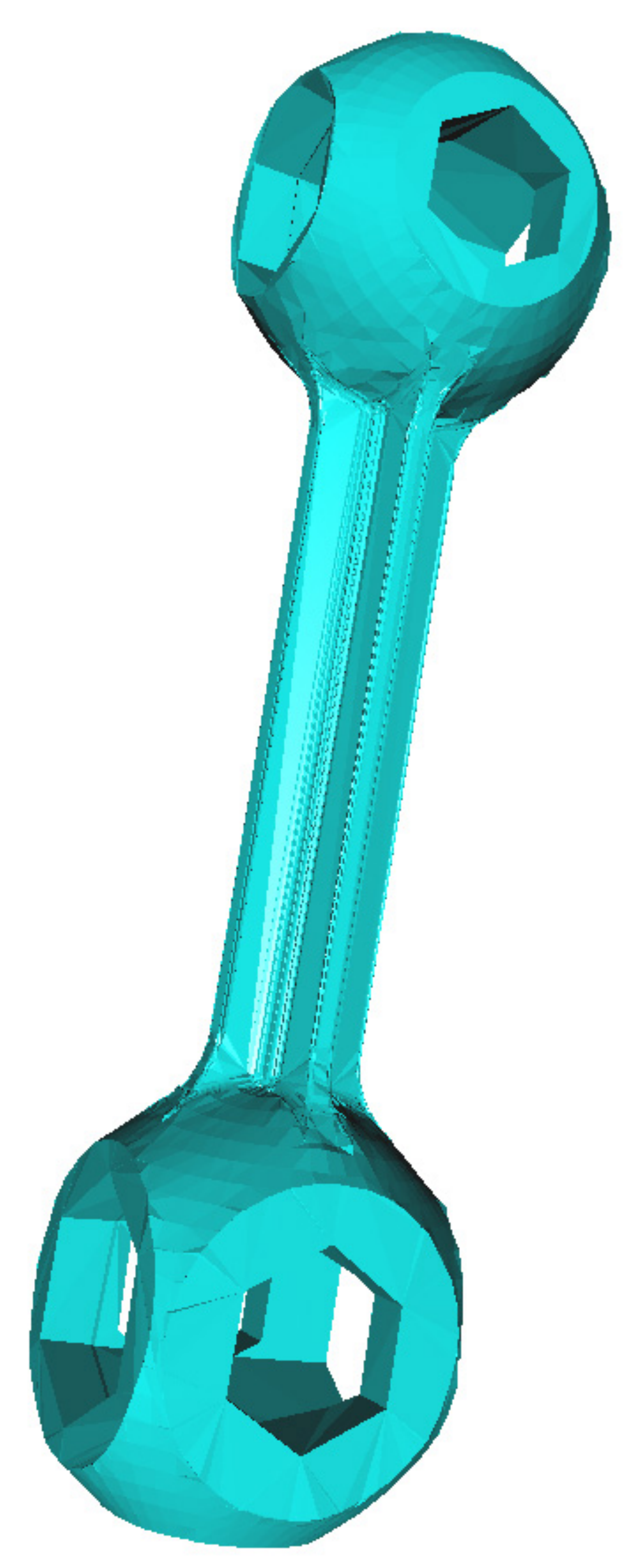}}
  \subfigure[]{\includegraphics[width=3.0cm]{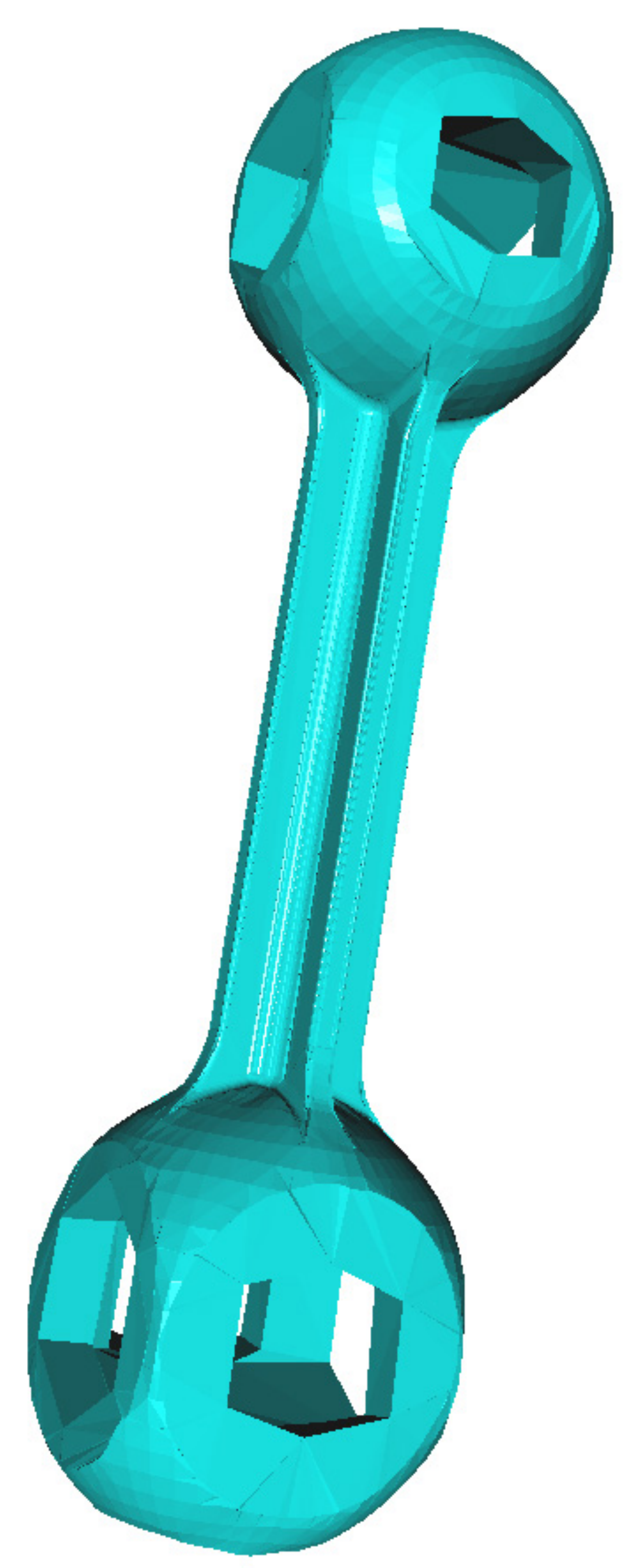}}
  \subfigure[]{\includegraphics[width=3.0cm]{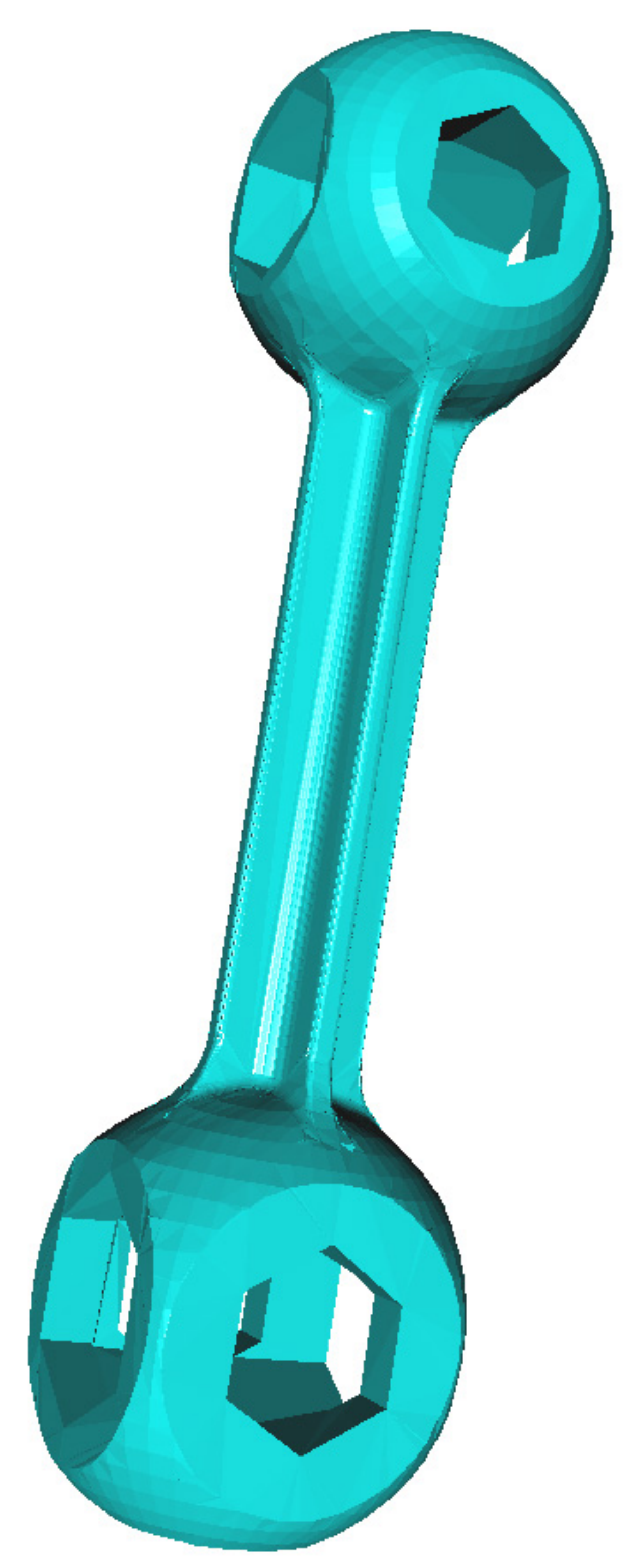}}
  \caption{(a) The ground truth model and its triangulation;
  (b) The mesh corrupted by uniformly distributed noise in normal directions (max. deviation $\pm10\%l_e$);
  (c) BMF; (d) BNF; (e) proposed($\mathbf{p}_i^\ast=\mathbf{p}_i$). }
  \label{Fig:Bonewrench_denoise}
\end{figure*}

Figure~\ref{Fig:Bonewrench_denoise} illustrates a mesh of a CAD model. The
original mesh contains highly irregular triangles and the surface is
consisting of piecewise smooth patches together with sharp features. From the
figure we can see that our proposed filter can almost recover the original
mesh exactly. As a comparison, there still exist flipped edges on the meshes
denoised by BMF or BNF due to the existence of long and thin triangles.

%\begin{figure*}[htb]
%  \centering
%  \subfigure[]{\includegraphics[width=3.8cm]{crank_original.eps}}
%  \subfigure[]{\includegraphics[width=3.8cm]{crank_original+wire.eps}}
%  \subfigure[]{\includegraphics[width=3.8cm]{crank_noisy.eps}}
%  \subfigure[]{\includegraphics[width=3.8cm]{crank_LBF_3itn.eps}}
%  \caption{(a)\&(b) the ground truth CAD model and the triangulation;
%  (c) the model corrupted by uniformly distributed noise in normal directions (max. deviation $\pm10\%l_e$);
%  (d) the result by 3 iterations of filtering using H-MLS ($\mathbf{p}_i^\ast=\mathbf{p}_i$).}\label{Fig:Filtering Crank_noisy}
%\end{figure*}

%Figure~\ref{Fig:Filtering Crank_noisy} illustrates another CAD model that contains long and thin triangles. The original mesh vertices have been corrupted by random noise in normal directions. A nearly exact surface mesh has been obtained after 3 iterations of H-LMS filtering. We note that many existing mesh denoising algorithms fail to filter noise from such a corrupted mesh with extreme triangulation. See the failure denoising examples reported in~\citep{HeSchaefer13:L0meshdenoising,ZhangWuZD15:variationMeshNormalDenoising,ZhangDZBL15:guidedNormalFiltering,ZhaoQin2018L0Meshdenoising}.

\begin{figure*}[htb]
  \centering
  \subfigure[]{\includegraphics[width=3.0cm]{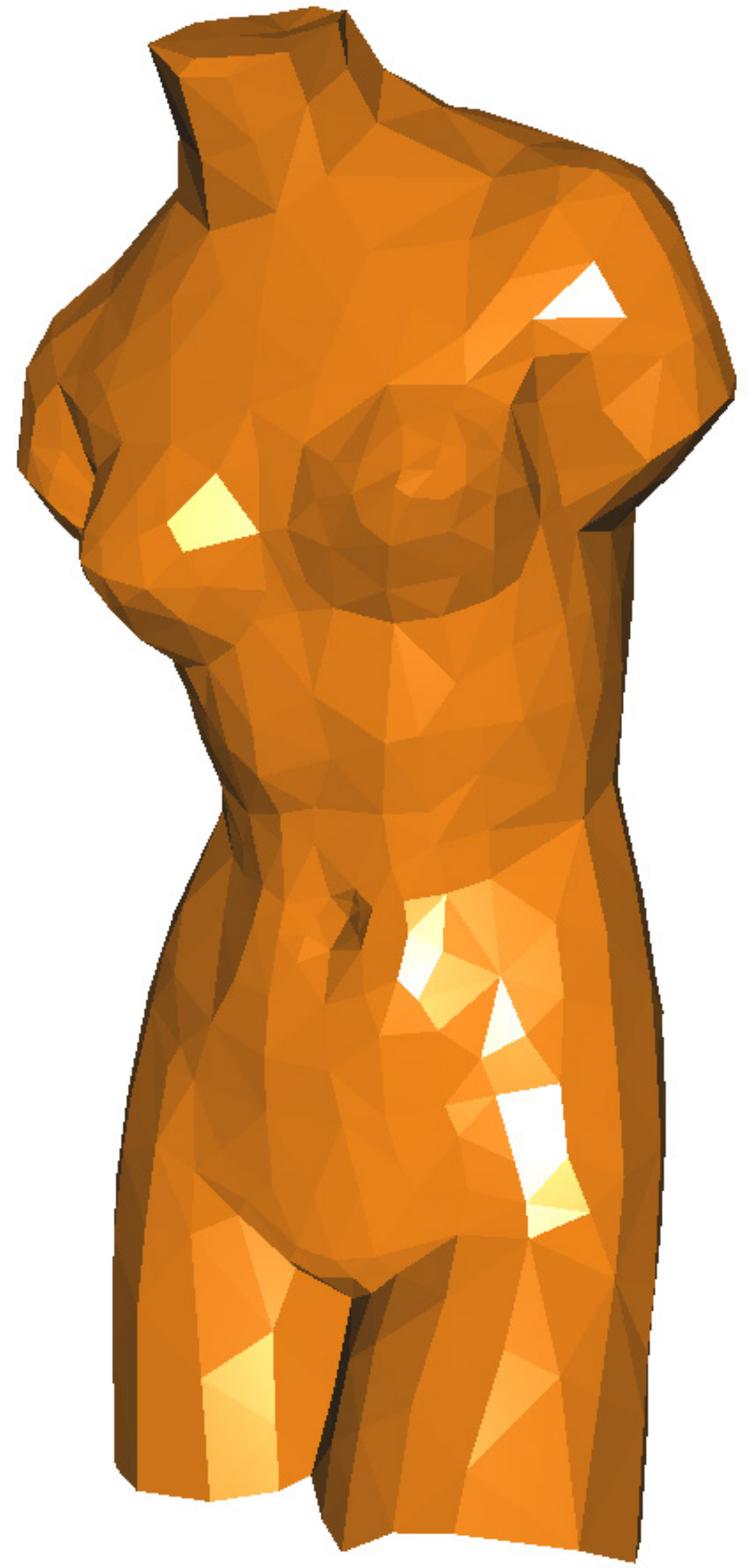}}
  \subfigure[]{\includegraphics[width=3.0cm]{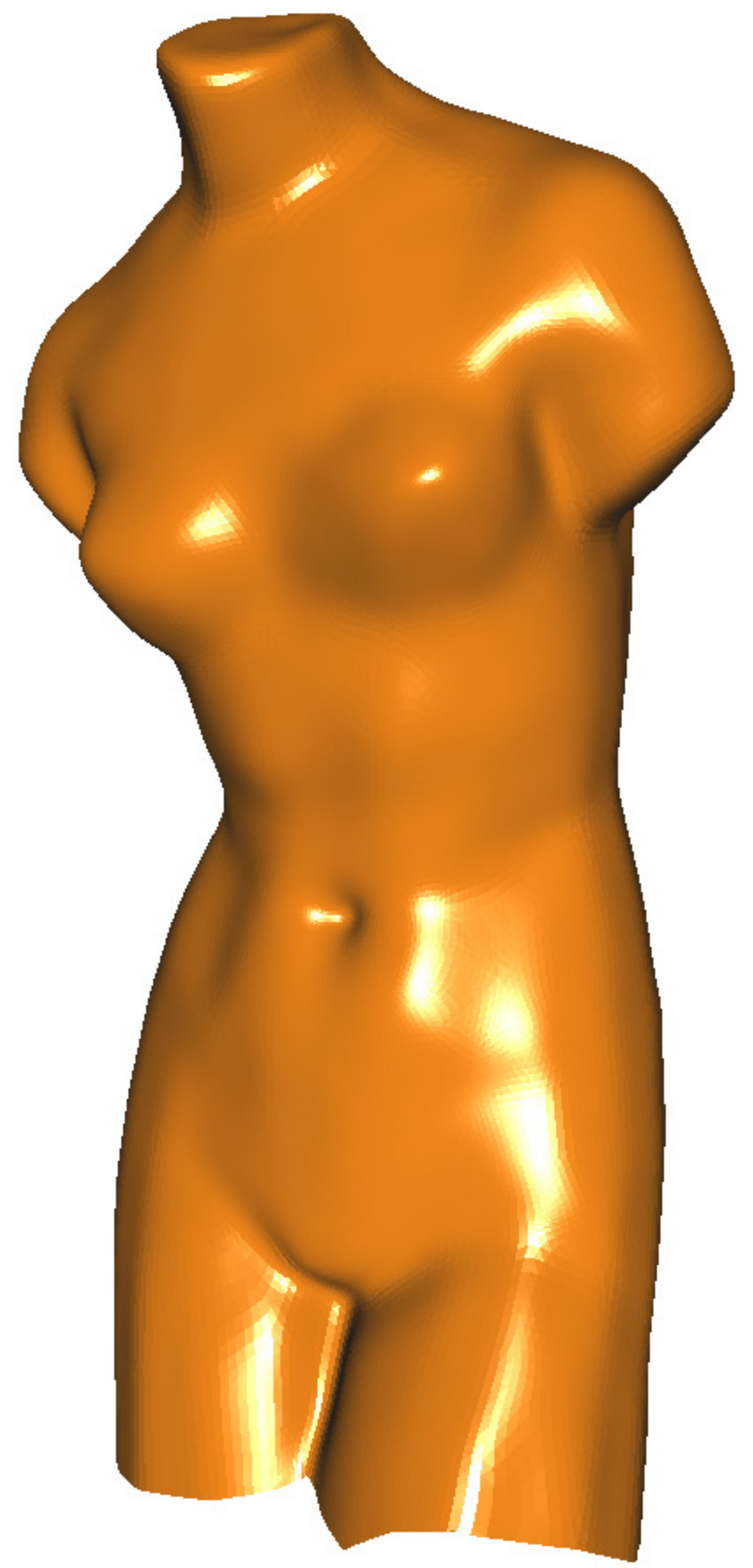}}
  \subfigure[]{\includegraphics[width=3.0cm]{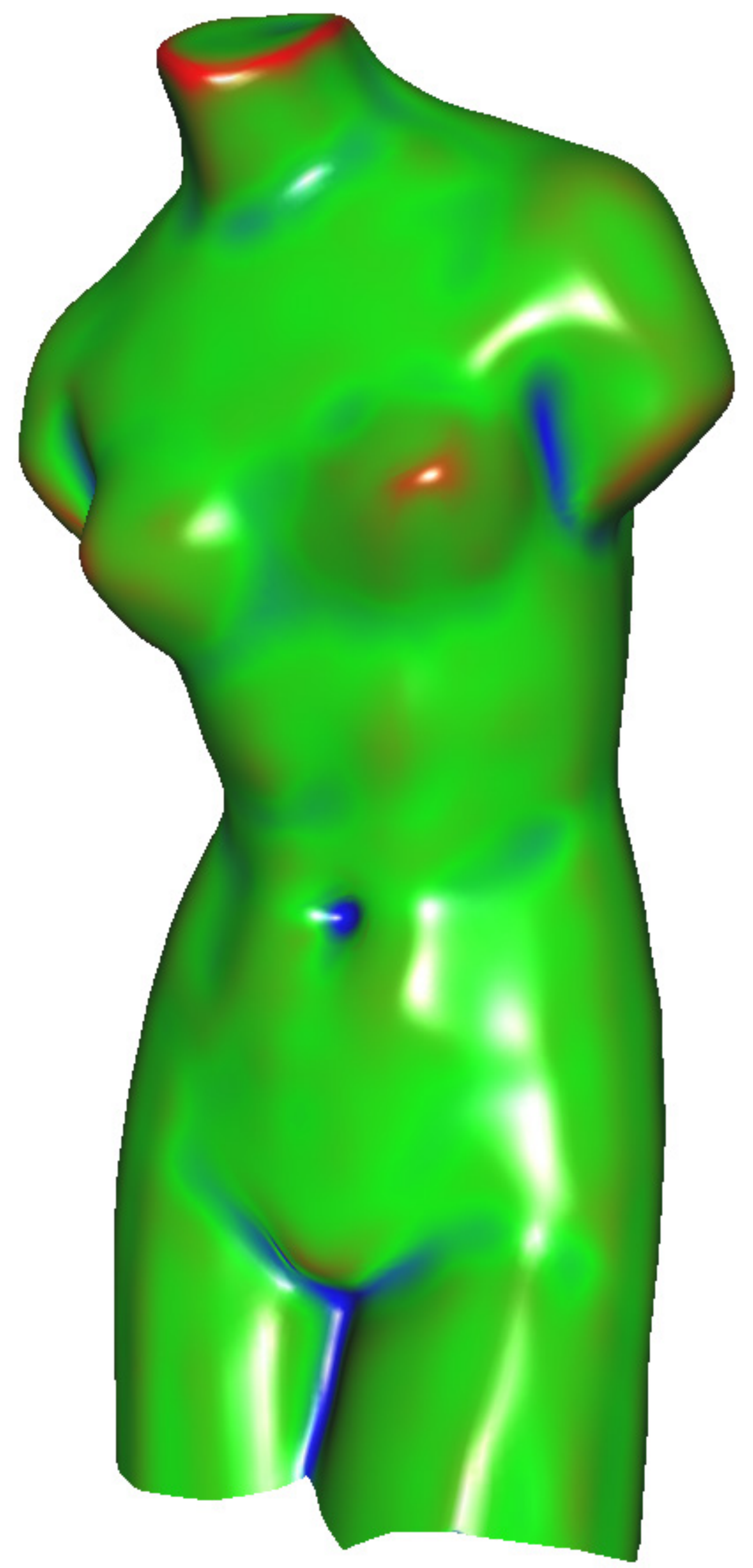}}
  \subfigure[]{\includegraphics[width=3.0cm]{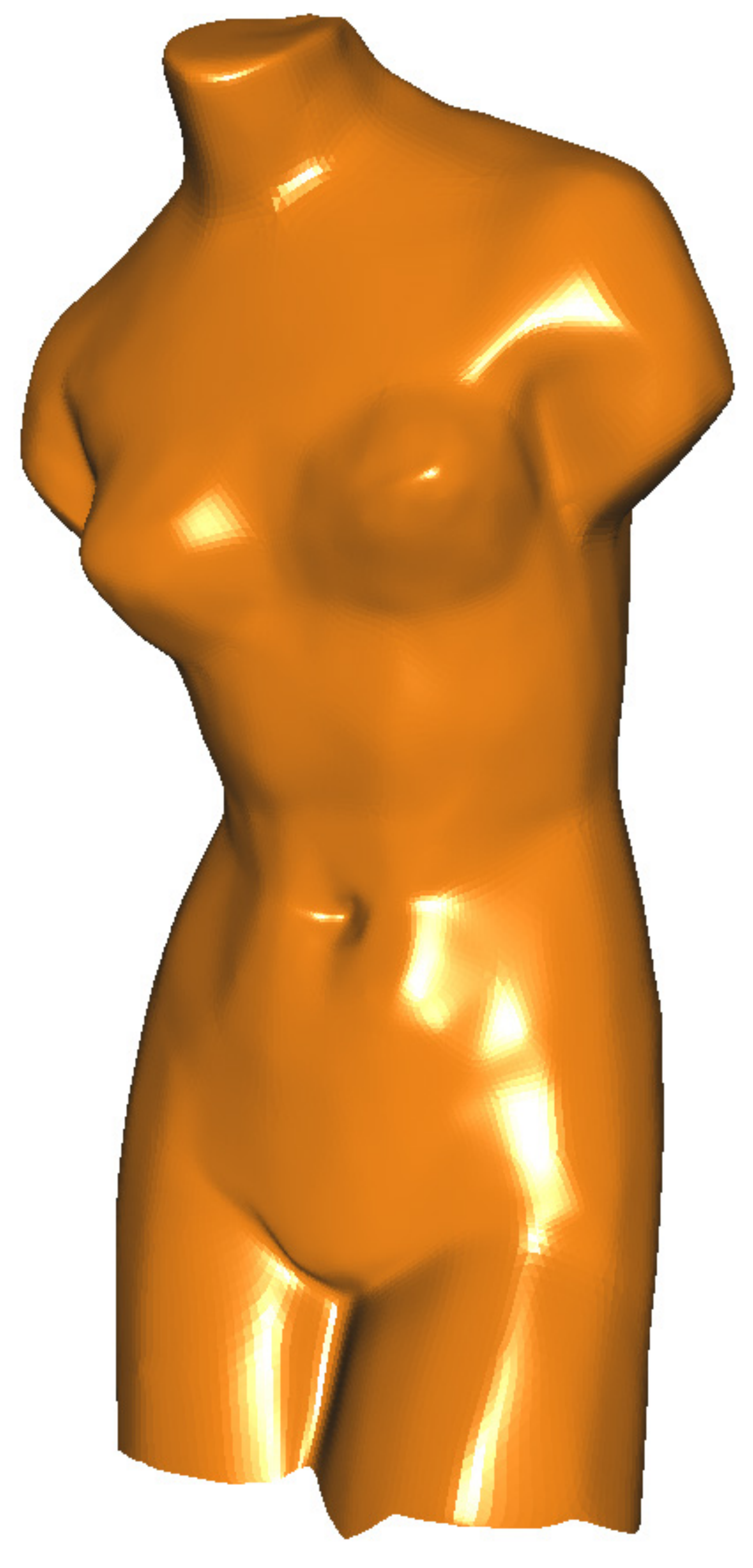}}
  \subfigure[]{\includegraphics[width=3.0cm]{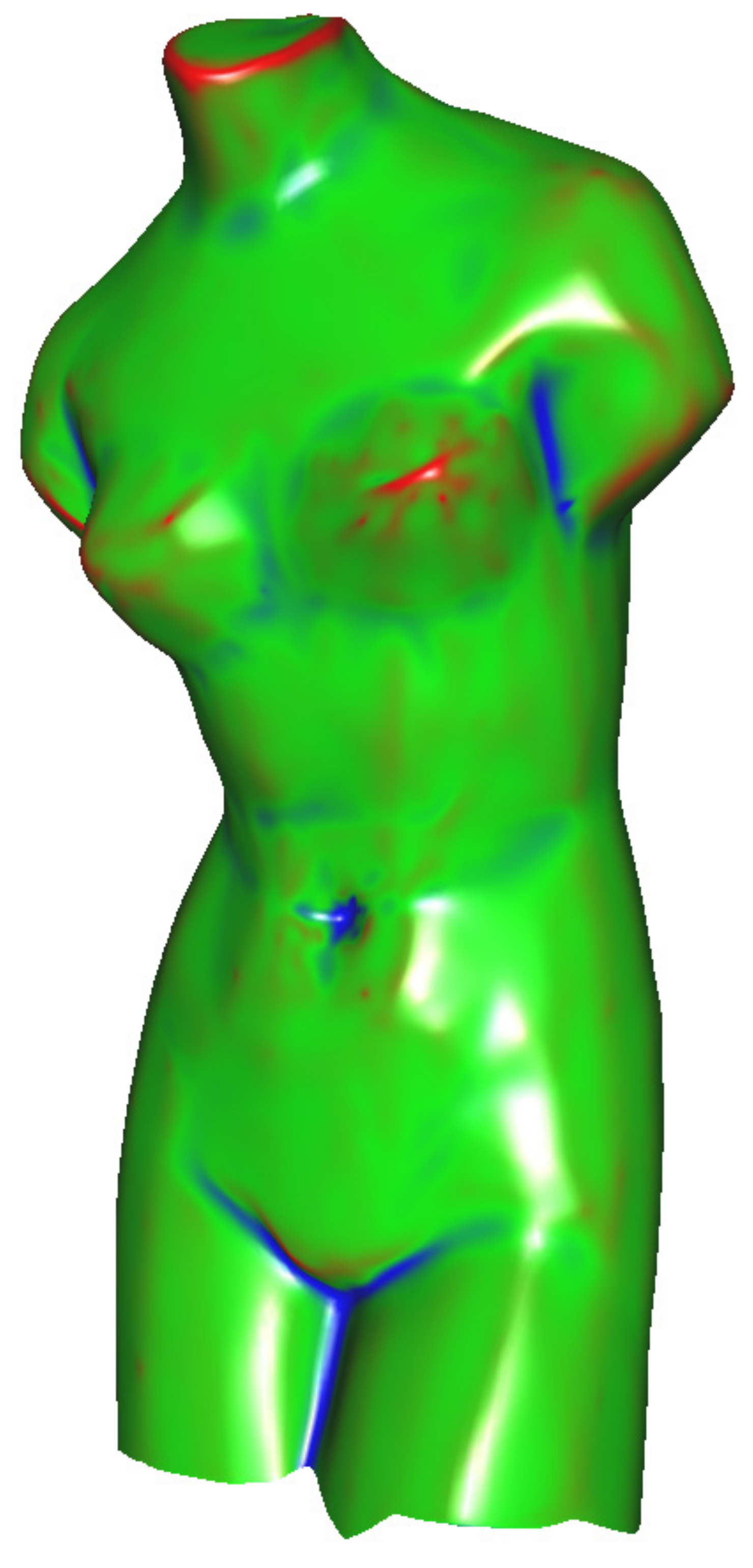}}
  \caption{Constructing a smooth surface by filtering a linearly subdivided mesh: (a) the mesh after 3 iterations of binary linear subdivision;
  (b) result by the proposed filter($\mathbf{p}_i^\ast=\mathbf{p}_i^{center}$); (c) the mean curvature plot of (b);
  (d) the mesh obtained by linear filtering; (e) mean curvature plot of (d). }
  \label{Fig:venus_smoothing}
\end{figure*}

Besides filtering a noisy mesh, our proposed filter can also be used to
construct smooth surfaces from a rough mesh.
Figure~\ref{Fig:venus_smoothing}(a) illustrates a triangular mesh obtained by
3 times of linear binary subdivision. A visually smooth surface is obtained
by 5 iterations of H-LMS filtering. The curvature plot also demonstrates
the quality of the smoothed mesh. See
Figure~\ref{Fig:venus_smoothing}(b)\&(c) for the result. As a comparison, if
we filter the linearly subdivided mesh by a linear filter, the obtained
surface is no longer as fair as expected. See
Figure~\ref{Fig:venus_smoothing}(d)\&(e) for the surface obtained by 10
iterations of $\lambda|\mu$ algorithm~\citep{Tau95}. We note that discrete
fair surfaces can also be constructed by some optimization techniques,
e.g.~\citep{SchneiderK2001GeoFairing,CranePS13Robustfairing}. Even though, our
proposed filter owns the advantages of simplicity and efficiency for fair
surface modeling.

%-------------------------------------------------------------------------
\section{Conclusions and discussion}
\label{Section:conclusion}

In this paper we have presented a simple but effective filter for
anisotropic mesh filtering. Particularly, new positions of mesh vertices are
computed by homogeneous least-squares fitting of moving constants to Hermite data.
Curvature-aware parameters and weights for the least squares fitting have been
robustly computed from the noisy data. The proposed filter can be used to
filter meshes with various kinds of noise as well as meshes with highly
irregular triangulation. The filtered mesh has high quality as well as high
fidelity to the original data. Smooth surfaces with various curvatures can be restored from noisy meshes effectively by the proposed filter. Sharp features with
discontinuous normals can also be preserved well when the noise magnitudes
are not high.

\textbf{Limitations}. Our proposed filter can be used to filter meshes with
locally high curvatures but free of sharp features or meshes having sharp
features but with only low magnitudes noise very well. If the sharp features
of a noisy mesh cannot be distinguished from high-magnitude noise locally,
our proposed filter no longer preserves sharp features.

\textbf{Future work}. At present we filter meshes with no need of filtering normals beforehand.
Combining the filter with other geometric processing techniques such as
robust normal estimation or feature detection, etc. will give more impressive
results. We focus on mesh filtering in this paper, another interesting future
work is to adapt the proposed filter for point set surface processing.

%%%%%%%%%%%%%%%%%%%%%%%%%%%%%%%%%%%%%%%%%%%%%%%%%%%%%%%%%%%%%%%%%%%%%
%\section*{References} % needed on some systems
\bibliography{HomogeneousFilter-bibliography}

\end{document}